\documentclass[11pt]{article}
\usepackage[letterpaper]{geometry}
\usepackage{amsmath,amsthm,amsfonts,amssymb}
\usepackage{enumerate,color,xcolor}
\usepackage{graphicx}
\usepackage{subfigure}
\usepackage{url}

\numberwithin{equation}{section}
\theoremstyle{plain}

\newtheorem{theorem}{Theorem}

\newtheorem{lemma}[theorem]{Lemma}

\newtheorem{remark}[theorem]{Remark}
\newtheorem{assumption}[theorem]{Assumption}

\begin{document}

\begin{center}
  \Large \bf Some asymptotic results for nonlinear Hawkes processes
\end{center}

\author{}
\begin{center}
{Fuqing Gao}\,\footnote{School of Mathematics and Statistics, Wuhan University, Wuhan 430072, People's Republic of China; fqgao@whu.edu.cn},
  Lingjiong Zhu\,\footnote{Department of Mathematics, Florida State University, 1017 Academic Way, Tallahassee, FL-32306, United States of America; zhu@math.fsu.edu.
  }
\end{center}

\begin{center}
 \today
\end{center}

\begin{abstract}
Hawkes process is a class of simple point processes with self-exciting
and clustering properties. Hawkes process has been widely applied 
in finance, neuroscience, social networks, criminology, seismology, and many
other fields. In this paper, we study fluctuations, large deviations and moderate deviations
nonlinear Hawkes processes in a new asymptotic regime, the large intensity function and the small exciting function regime. 
It corresponds to the large baseline intensity asymptotics for the linear case, 
and can also be interpreted as the asymptotics for the mean process of Hawkes processes on a large network. 
\end{abstract}

\section{Introduction}

Let $N$ be a simple point process on $\mathbb{R}$ and let $\mathcal{F}^{-\infty}_{t}:=\sigma(N(C),C\in\mathcal{B}(\mathbb{R}), C\subset(-\infty,t])$ be
an increasing family of $\sigma$-algebras. Any nonnegative $\mathcal{F}^{-\infty}_{t}$-progressively measurable process $\lambda_{t}$ with
\begin{equation*}
\mathbb{E}\left[N(a,b]|\mathcal{F}^{-\infty}_{a}\right]=\mathbb{E}\left[\int_{a}^{b}\lambda_{s}ds\big|\mathcal{F}^{-\infty}_{a}\right],
\end{equation*}
a.s. for all intervals $(a,b]$ is called an $\mathcal{F}^{-\infty}_{t}$-intensity of $N$. We use the notation $N_{t}:=N(0,t]$ to denote the number of
points in the interval $(0,t]$.

A Hawkes process is a simple point process $N$ admitting an $\mathcal{F}^{-\infty}_{t}$-intensity
\begin{equation}
\lambda_{t}:=\phi\left(\int_{-\infty}^{t-}h(t-s)N(ds)\right),\label{dynamics}
\end{equation}
where $\phi(\cdot):\mathbb{R}\rightarrow\mathbb{R}^{+}$ is locally integrable, left continuous,
$h(\cdot):\mathbb{R}^{+}\rightarrow\mathbb{R}$ and locally integrable.
In \eqref{dynamics}, $\int_{-\infty}^{t-}h(t-s)N(ds)$ stands for $\sum_{\tau<t}h(t-\tau)$, where
$\tau$ are the occurrences of the points before time $t$. In the literature, $h(\cdot)$ and $\phi(\cdot)$ are usually referred to
as exciting function (or sometimes kernel function or self-interaction function) and intensity function respectively, 
see e.g. \cite{Chevallier}.
A Hawkes process is linear if the intensity function $\phi(\cdot)$ is linear and it is nonlinear otherwise.

The Hawkes process when $\phi(\cdot)$ is linear was first proposed by Alan Hawkes
in 1971 to model earthquakes and their aftershocks \cite{Hawkes}. 
The nonlinear Hawkes process was first introduced by Br\'{e}maud and Massouli\'{e} \cite{Bremaud}.
The Hawkes process naturally generalizes the Poisson process and it
captures both the self-exciting property and the clustering effect, and it is a very versatile model
for statistical analysis. These explain why it has wide applications in neuroscience, genome analysis,
criminology, social networks, healthcare, seismology, insurance, finance and many other fields.
For a list of references, we refer to \cite{ZhuThesis}.

Most of the asymptotic results for Hawkes processes in the literature
are the large time limit theorems. 
For the linear Hawkes process,
the functional law of large numbers
and functional central limit theorems were studied in Bacry et al. \cite{Bacry};
the large deviations principle was studied in Bordenave
and Torrisi \cite{Bordenave}; and the moderate deviation principle was obtained in Zhu \cite{ZhuMDP}.
The precise large and moderate deviations are recently studied in Gao and Zhu \cite{GaoFZhu}.
For the nonlinear Hawkes process, Zhu \cite{ZhuCLT} studied the functional central limit theorems
by using Poisson embeddings and a careful analysis of the decay of the correlations over time. 
In \cite{ZhuII}, Zhu obtained a process-level, i.e. level-3 large deviation principle
and the rate function is expressed as a variational problem optimizing over
a certain entropy function of any simple point process against the underlying nonlinear Hawkes process.
When the exciting function is exponential and the process is Markovian, 
an alternative expression for the rate function
for the large deviations was obtained in Zhu \cite{ZhuI}.
Very recently, using the techniques as a combination of Poisson embeddings, Stein's method and Malliavin calculus, 
the quantitative Gaussian and Poisson approximations were studied in Torrisi \cite{TorrisiI,TorrisiII}.
The Malliavin calculus for Hawkes processes has also appeared in \cite{Takeuchi}.
In the case of linear Hawkes process, the limit theorems for nearly unstable, also known as, nearly critical case, 
that is, when $\phi(z)=\nu+z$ and $\Vert h\Vert_{L^{1}}\approx 1$ are studied
in Jaisson and Rosenbaum \cite{Jaisson} when the exciting function has light tail
and in Jaisson and Rosenbaum \cite{JaissonII} when the exciting function has heavy tail.

There have been some progress made in the direction of asymptotic results
other than the large time limits. For instance, when the exciting function is exponential,
the intensity process and the pair $(N_{t},\lambda_{t})$ are Markovian. 
In Gao and Zhu \cite{GZ}, they studied the functional central limit theorems
for the linear Hawkes process
when the initial intensity is large, and they further studied the large deviations
and applied their results to insurance and queueing systems in \cite{GZ2}.
For the more general linear and non-Markovian case, Gao and Zhu \cite{GZ3} considered the large
baseline intensity asymptotic results and studied the applications to queueing systems.

In recent years, the mean-field limit for high dimensional Hawkes processes has also been 
studied, and it first appeared in Delattre et al. \cite{Delattre}. They showed
that under a certain setting, the mean-field limit is an inhomogeneous Poisson process. 
Other mean-field limit works include Chevallier \cite{Chevallier} who
studied a generalized Hawkes process model with an inclusion of the dependence on
the age of the process, and Delattre and Fournier \cite{DF} who studied the mean-field limit
for Hawkes processes on a graph with two nodes whether or not influence each other modeled
by i.i.d. Bernoulli random variables.

In this paper, we are interested in studying a new asymptotic regime
for the nonlinear Hawkes process starting from empty past history, 
in which the intensity function is large and the exciting function is small. 
More precisely, we introduce the small parameter $\epsilon>0$ and consider
the nonlinear Hawkes process $N_{t}^{\epsilon}$ with intensity:
\begin{equation}
\lambda_{t}^{\epsilon}=\frac{1}{\epsilon}\phi\left(\int_{0}^{t-}\epsilon h(t-s)dN_{s}^{\epsilon}\right).
\end{equation}
In this asymptotic regime, the pair of the intensity function and the exciting function
has the transformation $(\phi,h)\mapsto(\frac{1}{\epsilon}\phi,\epsilon h)$. 

Now, let us explain why this asymptotic regime is natural and also point out that
such a regime has been studied extensively in many similar settings in the literature.

When the Hawkes process is linear, say $\phi(z)=\nu+z$, $h(\cdot):\mathbb{R}^{+}\rightarrow\mathbb{R}^{+}$,
where $\nu$ is the baseline intensity, we have
\begin{equation}
\lambda_{t}^{\epsilon}=\frac{\nu}{\epsilon}
+\int_{0}^{t-}h(t-s)dN_{s}^{\epsilon}.
\end{equation}
This gives the intensity of a linear Hawkes process
with exciting function $h$, 
and a large baseline intensity $\frac{\nu}{\epsilon}$.
Therefore, the asymptotic regime considered
in this paper corresponds to the large baseline intensity regime that is studied in Gao and Zhu \cite{GZ3}.

The asymptotic regime studied in this paper for the univariate Hawkes process
is also equivalent for the asymptotics for the mean process for the 
high-dimensional multivariate Hawkes process. Our work is related to the mean-field limit
for high-dimensional Hawkes processes in \cite{Delattre, Chevallier, DF}.
To see the connection of our work with the mean-field limit literature of Hawkes processes, 
let us first define a multivariate Hawkes process as follows.
An $N$-dimensional Hawkes process $(Z_{t}^{1},\ldots,Z_{t}^{n})$
is an $N$-dimensional point process 
admitting an $\mathcal{F}_{t}$-intensity $(\lambda_{t}^{1},\ldots,\lambda_{t}^{N})$
such that
\begin{equation}
\lambda_{t}^{i}:=\phi_{i}\left(\sum_{j=1}^{N}\int_{0}^{t-}h_{ij}(t-s)dZ_{s}^{j}\right),
\end{equation}
where $\phi_{i}(\cdot):\mathbb{R}\rightarrow\mathbb{R}^{+}$ is locally integrable, left continuous,
$h_{ij}(\cdot):\mathbb{R}^{+}\rightarrow\mathbb{R}$ and
we always assume that $\Vert h_{ij}\Vert_{L^{1}}=\int_{0}^{\infty}h_{ij}(t)dt<\infty$.
For the multivariate Hawkes process, a jump in one component will not only increase
the intensity of future jumps of its own component, known as the self-exciting property,
but also increase the intensity of the future jumps of or the other components
that are connected to its own component, which is known as the mutually-exciting property.
By using the Poisson embeddings, see e.g \cite{Bremaud,Delattre},
we can express the Hawkes process $(Z_{t}^{1},\ldots,Z_{t}^{n})$
as the solution of a Poisson driven SDE:
\begin{equation}
Z^{i}_t=\int_0^t \int_0^\infty {\bf 1}_{\left\{z \leq \phi_{i}\left(\sum_{j=1}^N\int_0^{s-}h_{ij}(s-u)dZ_u^{j}\right)\right\}}
\pi^i(ds\,dz),
\qquad
1\leq i\leq N,
\end{equation}
where $\{\pi^i(ds\,dz), i\geq 1\}$ are a sequence of  i.i.d.  Poisson
measures with common intensity measure $dsdz$ on
$[0,\infty) \times [0,\infty)$. 
As a special case, for each $N\geq 1$, we let $h_{ij}=\frac{1}{N}h$
and $\phi_{i}=\phi$ and
we consider the Hawkes process
$(Z^{N,1}_t,\dots,Z^{N,N}_t)_{t\geq 0}$ which can be expressed as
\begin{equation} \label{N-dim-Hawkes-process-eq}
Z^{N,i}_t=\int_0^t \int_0^\infty {\bf 1}_{\left\{z \leq \phi \left(N^{-1}\sum_{j=1}^N\int_0^{s-}h(s-u)dZ_u^{N,j}\right)\right\}}
\pi^i(ds\,dz).
\end{equation}
The \textit{mean process} of the Hawkes processes is defined by
$(Z^{N,1}_t,\dots,Z^{N,N}_t)_{t\geq 0}$:
\begin{equation} \label{Hawkes-mean-process-eq-0}
\overline{Z}^{N}_t=\frac{1}{N}\sum_{i=1}^N Z^{N,i}_t,~t\geq 0.
\end{equation}
It follows from \eqref{N-dim-Hawkes-process-eq} that
\begin{equation}\label{compare1}
\overline{Z}^{N}_t=\int_0^t \int_0^\infty {\bf 1}_{\left\{z \leq \phi \left(\int_0^{s-}h(s-u)d\overline{Z}_u^{N}\right)\right\}}
\frac{1}{N}\sum_{i=1}^{N}\pi^{i}(dsdz),
\end{equation}
where $\sum_{i=1}^{N}\pi^{i}(dsdz)$ is a Poisson measure on $[0,\infty)\times[0,\infty)$
with intensity $N$.

On the other hand, let us recall that the nonlinear Hawkes process 
$N_{t}^{\epsilon}$
with the intensity function $\frac{\phi}{\epsilon}$ and the exciting function $\epsilon h$
can be expressed via Poisson embedding as the unique strong solution 
to the following equation:
\begin{equation}
N_{t}^{\epsilon}=\int_{0}^{t}\int_{0}^{\infty}1_{[0,\frac{1}{\epsilon}\phi(\int_{0}^{s-}\epsilon h(s-u)dN_{u}^{\epsilon})]}
(z)\pi(dzds),
\end{equation}
where $\pi(dzds)$ is a Poisson random measure on $[0,\infty)\times[0,\infty)$ with intensity $1$.
In this paper, we are interested in the asymptotics for $Z_{t}^{\epsilon}:=\epsilon N_{t}^{\epsilon}$,
which satisfies the dynamics:
\begin{align}
Z_{t}^{\epsilon}
&:=\epsilon\int_{0}^{t}\int_{0}^{\infty}1_{[0,\frac{1}{\epsilon}\phi(\int_{0}^{s-}h(s-u)dZ_{u}^{\epsilon})]}(z)\pi(dzds)
\nonumber
\\
&=\int_{0}^{t}\int_{0}^{\infty}1_{[0,\phi(\int_{0}^{s-}h(s-u)dZ_{u}^{\epsilon})]}(z)\epsilon\pi^{\epsilon^{-1}}(dzds).
\label{Hawkes-sde}
\end{align}
where $\pi^{\epsilon^{-1}}(dzds)$ is a Poisson random measure on $[0,\infty)\times[0,\infty)$ with intensity $\epsilon^{-1}$.

By comparing \eqref{compare1} with \eqref{Hawkes-sde},
it becomes clear that the mean process of an $N$-dimensional Hawkes process
defined in \eqref{N-dim-Hawkes-process-eq} 
has the same dynamics as a univariate Hawkes process
with $N=\frac{1}{\epsilon}$. 
All the asymptotic results we are going to derive in this paper for 
the $Z_{t}^{\epsilon}$ process automatically hold for the mean process $\overline{Z}_{t}^{N}$.
We will go back to this in Section \ref{SecMeanProcess}.

The asymptotic results for the mean process for a high-dimensional Hawkes process
in Section \ref{SecMeanProcess} can shed some lights for the applications of high-dimensional Hawkes processes
in various context. 
Hawkes processes have been applied to the study of neuroscience, 
see e.g. neuroscience, see e.g. \cite{PerniceI,PerniceII, RRT, Reynaud}.
More recently, mean-field limits for extended Hawkes processes
have been used to model the neural networks in e.g. \cite{Chevallier,DL,CDLO}.
The large deviations results in Section \ref{SecMeanProcess} can be used
to estimate the probability of rare events in a neural network. 
The moderate deviations results in Section \ref{SecMeanProcess} can be used to fill
in the gap between the second-order fluctuations and the large deviations regime.
We can also use the multivariate Hawkes process of dimension $N$ to represent
the loss process for $N$ firms in a large portfolio. The results in Section \ref{SecMeanProcess}
can be used to provide estimates for the tail probabilities for the loss of a large portfolio.
We refer to \cite{DaiPra,DDD,GSSS} for the works of large portfolio losses in finance.
Note that the results we obtained in Section \ref{SecMeanProcess} are
for the standard multivariate nonlinear Hawkes processes. In order to apply our results
to neural networks in neuroscience, large portfolio losses in finance, and many other contexts, one needs to extend our results
for the generalized Hawkes processes suitable for the applications in various contexts. Since there are many different
ways to generalize the standard multivariate nonlinear Hawkes processes for the purpose of applications, 
we restrict the study in this paper to the most standard nonlinear Hawkes processes. 
Nevertheless, the methodology presented in this paper should be applicable for various extensions.

The scalings in \eqref{Hawkes-sde} for stochastic equations with Poisson noise
have been widely studied in the literature, see e.g.
Budhiraja et al. \cite{Budhiraja}, 
Budhiraja et al. \cite{BudhirajaII}, 
Budhiraja et al. \cite{BudhirajaIII}.
The large deviations and moderate deviations for stochastic equations with jumps can be established usually using the variational representation  in  \cite{Budhiraja}. However, in our case, the coefficient of the dynamics  \eqref{Hawkes-sde}  is  a indicator function with path-dependency, which is  not continuous.  
As a result, we cannot apply the results from \cite{Budhiraja,BudhirajaII} directly.
Instead of pursuing a modification of the variational
representation approach in \cite{Budhiraja,BudhirajaII}, 
we will adopt a more direct approach to establish large and moderate deviations in our paper.

We organize this paper as follows. In Section \ref{MainSection},
we introduce the main results of the paper. We will study
fluctuations in Section \ref{CLTSection}, large deviations
in Section \ref{LDPSection} and moderate deviations in Section \ref{MDPSection}.
The asymptotic results for the mean process for a high-dimensional Hawkes process
are presented in Section \ref{SecMeanProcess}.
Finally, all the proofs will be given in Section \ref{ProofsSection}.

%%%%%%%%%%%%%%%%%%%%%%%%%%%%%%%%%%%%%%%%%%%%%%%%%%%%%%%%%%%%%%%%%%%%%%%%%%%%%%%%%%%%%%%%%%%%%%%
\section{Main Results}\label{MainSection}

Before we proceed, let us summarize here a list of key assumptions that
will be used throughout the paper.

\begin{assumption}\label{AssumpI}
$\phi(\cdot):\mathbb{R}\rightarrow\mathbb{R}^{+}$ 
is $\alpha$-Lipschitz for some $0<\alpha<\infty$. 
$h(\cdot):\mathbb{R}_{\geq 0}\rightarrow\mathbb{R}$ is locally integrable
and locally bounded. 
\end{assumption}

\begin{assumption}\label{AssumpII}
$h$ is differentiable and $|h'|$ is locally integrable.
\end{assumption}

\begin{assumption}\label{AssumpIII}
$\phi(\cdot)$ is $\alpha$-Lipschitz and $\alpha\Vert h\Vert_{L^{1}[0,T]}=\alpha\int_{0}^{T}|h(t)|dt<1$.
\end{assumption}

\begin{assumption}\label{AssumpIV}
$\phi(\cdot)$ is twice differentiable and $\Vert\phi''\Vert_{L^{\infty}}=\sup_{x\geq 0}|\phi''(x)|<\infty$.
\end{assumption}

\begin{assumption}\label{AssumpV}
$\inf_{x\geq 0}\phi(x)>0$,
$h$ is differentiable and $\Vert h'\Vert_{L^{\infty}[0,T]}=\sup_{t\in[0,T]}|h'(t)|<\infty$.
\end{assumption}

We collect here a set of notations that will be used throughout the paper.

\begin{itemize}
\item
$C[0,T]$ is the space of real-valued continuous functions on $[0,T]$;
\item 
$D[0,T]$ is the space of real-valued c\`{a}dl\`{a}g functions on $[0,T]$ equipped with Skorokhod topology;
\item
$\mathcal{AC}_{0}[0,T]$ is the space of functions
$f:[0,T]\rightarrow\mathbb{R}$ that are absolutely continuous
with $f(0)=0$;
\item
$\mathcal{AC}_{0}^{+}[0,T]$ is the space of non-decreasing functions 
$f:[0,T]\rightarrow\mathbb{R}$ that are absolutely continuous
with $f(0)=0$.
\end{itemize}

%%%%%%%%%%%%%%%%%%%%%%%%%%%%%%%%%%%%%%%%%%%%%%%%%%%%%%%%%%%%%%%%%%%%
\subsection{Fluctuations}\label{CLTSection}

In this section, we are interested to study the fluctuations
of $Z^{\epsilon}$ around its limit $Z^{0}$. We will obtain a functional
central limit theorem for $Z^{\epsilon}$.

As $\epsilon\rightarrow 0$, $Z_{t}^{\epsilon}$ will converge
on $D[0,T]$ to a deterministic function $Z_{t}^{0}$ that satisfies the equation:
\begin{equation}\label{Z0Eqn}
Z_{t}^{0}=\int_{0}^{t}\phi\left(\int_{0}^{s}h(s-u)dZ_{u}^{0}\right)ds.
\end{equation}
Indeed, this result will follow from the fluctuation result
for $Z_{t}^{\epsilon}$, that is, we will show
that $\frac{Z_{t}^{\epsilon}-Z_{t}^{0}}{\sqrt{\epsilon}}$ converges
in distribution on $D[0,T]$ to a non-trivial stochastic limit, which turns out to be a continuous Gaussian process. 
Let us notice that the equation \eqref{Z0Eqn} has a unique locally bounded and non-negative solution under
certain assumptions, see Delattre \cite{Delattre}. 
It is interesting that the mean of the inhomogeneous Poisson process
as the mean-field limit for high dimensional
Hawkes processes leads to the same limiting equation as in \eqref{Z0Eqn}.

Let us define:
\begin{equation}
X_{t}^{\epsilon}
=\frac{Z_{t}^{\epsilon}-Z_{t}^{0}}
{\sqrt{\epsilon}}.
\end{equation}
 
\begin{theorem}\label{CLTThm}
Suppose Assumption \ref{AssumpI}, Assumption \ref{AssumpII} and Assumption \ref{AssumpIV} hold. 
$X^\epsilon$ converges in distribution on $D[0,T]$ to a continuous Gaussian process $X_t$ defined by
\begin{equation}\label{Gauss-process}
\begin{aligned}
X_t=&\int_{0}^{t}\phi'\left(\int_{0}^{s}h(s-u)dZ_{u}^{0}\right)\left(h(0)X_s+\int_{0}^{s}X_uh'(s-u)du\right)ds \\
&+ \int_{0}^{t}\sqrt{ \phi\left(\int_{0}^{s}h(s-u)dZ_{u}^{0}\right)}dW_{s},
\end{aligned}
\end{equation}
where $W_{t}$ is a standard Brownian motion.
\end{theorem}

\begin{remark} 
The Gaussian process defined by \eqref{Gauss-process} is also 
a semimartingale and 
$$
h(0)X_s+\int_{0}^{s}X_uh'(s-u)du=\int_{0}^{s}h(s-u)dX_u.
$$ 
Thus, the Gaussian process $X_t$ has the following equivalent characterization:
\begin{equation}
\begin{aligned}
X_t=&\int_{0}^{t}\phi'\left(\int_{0}^{s}h(s-u)dZ_{u}^{0}\right)\int_{0}^{s}h(s-u)dX_uds  + \int_{0}^{t}\sqrt{ \phi\left(\int_{0}^{s}h(s-u)dZ_{u}^{0}\right)}dW_{s}.
\end{aligned}
\end{equation}
\end{remark}

A key component of the proof of Theorem \ref{CLTThm}
is the tightness of the sequence $X^\epsilon$ on $D[0,T]$
that we will establish in the following lemma. 

\begin{lemma}\label{Xtight}
Suppose Assumption \ref{AssumpI} and Assumption \ref{AssumpII} hold, 
$X_{t}^{\epsilon}$ is tight on $D[0,T]$ and the all limits are in $C[0,T]$.
\end{lemma}

The proof of the tightness of $X^\epsilon$, relies on two auxiliary lemmas. 
The first lemma, i.e. Lemma \ref{Z1Bound}
gives a uniform bound on the first moment of $Z_{T}^{\epsilon}$, uniformly in $\epsilon$,
and the second lemma, i.e. Lemma \ref{X2Bound}, gives us a uniform bound
on the second moment of the running maximum of $X^\epsilon$ process, uniformly in $\epsilon$.

\begin{lemma}\label{Z1Bound}
Suppose Assumption \ref{AssumpI} holds.
\begin{equation}
\sup_{\epsilon>0}\mathbb{E}[Z_{T}^{\epsilon}]\leq
\phi(0)Te^{\alpha\Vert h\Vert_{L^{\infty}[0,T]}T}.
\end{equation}
\end{lemma}

\begin{lemma}\label{X2Bound}
Suppose Assumption \ref{AssumpI} and Assumption \ref{AssumpII} hold, 
\begin{equation}
\sup_{\epsilon>0}\mathbb{E}\left[\sup_{0\leq t\leq T}(X_{t}^{\epsilon})^{2}\right]<\infty.
\end{equation}
\end{lemma}

The proofs of Theorem \ref{CLTThm}, Lemma \ref{Xtight}, Lemma \ref{Z1Bound} and Lemma \ref{X2Bound}
will all be given in Section \ref{ProofsSection}.

\begin{remark}
Note that \cite{ZhuCLT} studied the large time fluctuations for stationary nonlinear Hawkes processes
and more precisely, as a special case for the linear Hawkes process $\phi(x)=\nu+x$,  
we have $\frac{N_{nt}-\mu t}{\sqrt{n}}\rightarrow \sigma B(t)$ in distribution on $D[0,T]$ as $n\rightarrow\infty$, where 
$\mu=\frac{\nu}{1-\Vert h\Vert_{L^{1}}}$ and $\sigma^{2}=\frac{\nu}{(1-\Vert h\Vert_{L^{1}})^{3}}$, 
and $B(t)$ is a standard Brownian motion.
Note that for the large time functional central limit theorem, the limiting variance
depends on $\Vert h\Vert_{L^{1}}$ only, while in our Theorem \ref{CLTThm}, it 
depends on the entire exciting function $h(t)$ for $t\in[0,T]$. Moreover, in our limit, we obtain a 
Gaussian process that in general is not a Brownian motion. 
\end{remark}

%%%%%%%%%%%%%%%%%%%%%%%%%%%%%%%%%%%%%%%%%%%%%%%%%%%%%%%%%%%%%%%%%%%%
\subsection{Large deviations}\label{LDPSection}

We have already seen that $Z_{t}^{\epsilon}$ converges to the limit $Z_{t}^{0}$
on $D[0,T]$ and have studied the fluctuations around this limit. 
It is natural to ask about the probability of the rare events
that the process $Z_{t}^{\epsilon}$ deviates away from its deterministic limit.
That is the question of large deviations in probability theory. 

We start by giving a formal definition of the large deviation principle. 
We refer to Dembo and Zeitouni \cite{Dembo} and Varadhan \cite{VaradhanII} 
for general background of large deviations and the applications. 

A sequence $(P_{\epsilon})_{\epsilon\in\mathbb{R}^{+}}$ of probability measures on a topological space $X$ 
satisfies the large deviation principle with rate function $I:X\rightarrow\mathbb{R}$ and speed $b(\epsilon)$ 
if $I$ is non-negative, 
lower semicontinuous and for any Borel set $A$, we have
\begin{equation}
-\inf_{x\in A^{o}}I(x)\leq\liminf_{\epsilon\rightarrow 0}\frac{1}{b(\epsilon)}\log P_{\epsilon}(A)
\leq\limsup_{\epsilon\rightarrow 0}\frac{1}{b(\epsilon)}\log P_{\epsilon}(A)\leq-\inf_{x\in\overline{A}}I(x).
\end{equation}
Here, $A^{o}$ is the interior of $A$ and $\overline{A}$ is its closure. 

Now, we are ready to state the main results of large deviations for $Z_{t}^{\epsilon}$ on $D[0,T]$.

\begin{theorem}\label{LDPThm}
Suppose Assumption \ref{AssumpI}, Assumption \ref{AssumpIII}
and Assumption \ref{AssumpV} hold.
Then, $\mathbb{P}(Z_{t}^{\epsilon}\in\cdot)$ satisfies a large deviation
principle on $D[0,T]$ equipped with Skorokhod topology
with the speed $\epsilon^{-1}$ and the rate function
\begin{equation}\label{IEqn}
I(\eta):=\int_{0}^{T}\ell\left(\eta'(t);\phi\left(\int_{0}^{t}h(t-s)d\eta(s)\right)\right)dt,
\end{equation}
if $\eta\in\mathcal{AC}_{0}^{+}[0,T]$ and $+\infty$ otherwise, where
\begin{equation}\label{ellEqn}
\ell(x;y):=x\log\left(\frac{x}{y}\right)-x+y.
\end{equation}
\end{theorem}

Instead of establishing a full large deviation principle in Theorem \ref{LDPThm} directly, 
our strategy is to first prove a local large deviation principle in Theorem \ref{localLDPThm}, with the main tool
being the change of measure technique for simple point processes. 
We then establish the exponential tightness in order to obtain a full large deviation principle.

%%%%%%%%%%%%%%%%%%%%%%%%%%%%%%%%%%%%%%%%%%%%%%%%%%%

We have the following local large deviation principle.

\begin{theorem}\label{localLDPThm}
Suppose Assumption \ref{AssumpI}
and Assumption \ref{AssumpV} hold.
For any $\eta\in D[0,T]$, 
\begin{equation}
\lim_{\delta\rightarrow 0}\lim_{\epsilon\rightarrow 0}
\epsilon\log
\mathbb{P}\left(\sup_{0\leq t\leq T}|Z_{t}^{\epsilon}-\eta(t)|\leq\delta\right)
=-I(\eta),
\end{equation}
where $I(\eta)$ is defined in \eqref{IEqn}.
\end{theorem}

Next, let us establish the exponential tightness
of the sequence $Z_{t}^{\epsilon}$ on $D[0,T]$.
The following Lemma \ref{exptightI} and Lemma \ref{exptightII},
together with the local large deviation principle will provide
us the full large deviation principle that is desired.

\begin{lemma}\label{exptightI}
Suppose Assumption \ref{AssumpI}, Assumption \ref{AssumpIII} hold. Then,
\begin{equation}
\limsup_{K\rightarrow\infty}\limsup_{\epsilon\rightarrow 0}\epsilon\log\mathbb{P}(Z_{T}^{\epsilon}\geq K)
=-\infty.
\end{equation}
\end{lemma}

\begin{lemma}\label{exptightII}
Suppose Assumption \ref{AssumpI}, Assumption \ref{AssumpIII} hold.
For any $\delta>0$,
\begin{equation}
\limsup_{M\rightarrow\infty}\limsup_{\epsilon\rightarrow 0}\epsilon\log\mathbb{P}\left(\sup_{0\leq s\leq t\leq T,
|t-s|\leq\frac{1}{M}}|Z_{t}^{\epsilon}-Z_{s}^{\epsilon}|\geq\delta\right)
=-\infty.
\end{equation}
\end{lemma}

\begin{remark} 
Theorem \ref{localLDPThm}, Lemma \ref{exptightI} and Lemma \ref{exptightII}  provide
actually the large deviation principle for $Z_{t}^{\epsilon}$ with respect to the uniform topology on $D[0,T]$, see e.g. Lemma A.1 in \cite{DGWU}, or Theorem 4.14 \cite{FK}.
\end{remark}

\begin{remark}
In \cite{Bordenave}, they obtained a sample path large deviation principle
for the large time scaling for Poisson cluster processes. More precisely,
the linear Hawkes process with $\phi(x)=\nu+x$, as a special case of the Poisson cluster process, has
the sample path large deviation principle that $\mathbb{P}(\frac{N_{n\cdot}}{n}\in\cdot)$ satisfies
a large deviation principle on $D[0,T]$ equipped with the topology of point-wise convergence
with the speed $n$ and the rate function $\int_{0}^{T}\mathcal{I}(f'(t))dt$ if $f\in\mathcal{AC}_{0}[0,T]$ and $+\infty$
otherwise, where
\begin{equation}\label{largetimeLDP}
\mathcal{I}(x):=x\log\left(\frac{x}{\nu+x\Vert h\Vert_{L^{1}}}\right)-x+x\Vert h\Vert_{L^{1}}+\nu,
\end{equation}
for $x\geq 0$ and $+\infty$ otherwise.
Note that since the assumption (37) in \cite{Bordenave} is not satisfied for the linear Hawkes process,
their large deviations results apply to the topology of point-wise convergence, but not the uniform topology.
Our results in Theorem \ref{LDPThm} differ in two ways. First, our rate function depends on the entire
function $h(t)$, $0\leq t\leq T$, rather than $\Vert h\Vert_{L^{1}}$ as in \eqref{largetimeLDP}. Second, we allow uniform topology
for the sample path large deviation principle.
\end{remark}

%%%%%%%%%%%%%%%%%%%%%%%%%%%%%%%%%%%%%%%%%%%%%%%%%%%%%%%%%%%%%%%%%%%%

\subsection{Moderate Deviations}\label{MDPSection}

In this section, we are interested in the moderate deviations
for $Z_{t}^{\epsilon}$. The moderate deviation principle fills
in the gap between the central limit theorem and the large deviation principle.
For a brief introduction to moderate deviations, we refer to Chap. 3.7. in Dembo and Zeitouni \cite{Dembo}.

Our approach to the proof of the moderate deviations is similar
to that of the large deviations. That is, we first establish
a local moderate deviation principle by using the change of measure technique, i.e. Theorem \ref{localMDPThm},
and then establish the appropriate exponential tightness estimates, i.e. Lemma \ref{MDPexptightI} and Lemma \ref{MDPexptightII}.

Our main result is the following:

\begin{theorem}\label{MDPThm}
Suppose Assumption \ref{AssumpI}, Assumption \ref{AssumpII}, Assumption \ref{AssumpIII},
Assumption \ref{AssumpIV} and Assumption \ref{AssumpV} hold.
Let $a(\epsilon)$ be a positive sequence such that $a(\epsilon),\frac{\epsilon}{a(\epsilon)^{2}}\rightarrow 0$
as $\epsilon\rightarrow 0$.
Then, $\mathbb{P}(\frac{Z_{t}^{\epsilon}-Z_t^0}{a(\epsilon)}\in\cdot)$ satisfies a large deviation
principle on $D[0,T]$ equipped with Skorokhod topology
with speed $\frac{a(\epsilon)^{2}}{\epsilon}$ and the rate function
\begin{equation}\label{JEqn}
J(\eta):=\frac{1}{2}\int_{0}^{T}\frac{\left(\eta'(t)-\phi'\left(\int_{0}^{t}h(t-u)dZ_{u}^{0}\right)
\int_{0}^{t}h(t-u)d\eta_{u}\right)^2}{\phi\left(\int_{0}^{t}h(t-u)dZ_{u}^{0}\right)} dt,
\end{equation}
if $\eta\in\mathcal{AC}_{0}[0,T]$ and $+\infty$ otherwise.
\end{theorem}

We first establish a local moderate deviation principle:

%%%%%%%%%%%%%%%%%%%%%%%%%%%%%%%%%%%%%%%%%%%%%%%%%
\begin{theorem}\label{localMDPThm}
Suppose Assumption \ref{AssumpI}, 
Assumption \ref{AssumpIV} and Assumption \ref{AssumpV} hold.
For any $\eta\in D[0,T]$,
$$\begin{aligned}
&\lim_{\delta\rightarrow 0}\lim_{\epsilon\rightarrow 0}
\frac{\epsilon}{a^2(\epsilon)}\log
\mathbb{P}\left(\sup_{0\leq t\leq T}\left|\frac{Z_{t}^{\epsilon}-Z_t^0}{a(\epsilon)}-\eta_t\right|\leq\delta\right)
=-J(\eta),
\end{aligned}
$$
where $J(\eta)$ is given in \eqref{JEqn}.
\end{theorem}
%%%%%%%%%%%%%%%%%%%%%%%%%%%%%%%%%%%%%%%%%%%%%%%%%%%%%%%%%%%%%
Next, we establish the exponential tightness
of the sequence $\frac{Z_{t}^{\epsilon}-Z_{t}^{0}}{a(\epsilon)}$ on $D[0,T]$
in the following lemmas.

\begin{lemma}\label{MDPexptightI}
Suppose Assumption \ref{AssumpI}, Assumption \ref{AssumpII}, Assumption \ref{AssumpIII} hold.
\begin{equation}
\limsup_{K\rightarrow\infty}\limsup_{\epsilon\rightarrow 0}
\frac{\epsilon}{a(\epsilon)^{2}}\log\mathbb{P}\left(\sup_{0\leq t\leq T}|Z_{t}^{\epsilon}-Z_{t}^{0}|\geq Ka(\epsilon)\right)
=-\infty.
\end{equation}
\end{lemma}

\begin{lemma}\label{MDPexptightII}
Suppose Assumption \ref{AssumpI}, Assumption \ref{AssumpII}, Assumption \ref{AssumpIII} hold.
For any $\delta>0$,
\begin{equation}
\limsup_{M\rightarrow\infty}\limsup_{\epsilon\rightarrow 0}
\frac{\epsilon}{a(\epsilon)^{2}}\log\mathbb{P}\left(\sup_{0\leq s\leq t\leq T,
|t-s|\leq\frac{1}{M}}|Z_{t}^{\epsilon}-Z_{s}^{\epsilon}-Z_{t}^{0}+Z_{s}^{0}|\geq\delta a(\epsilon)\right)
=-\infty.
\end{equation}
\end{lemma}

\begin{remark}
(i). Theorem \ref{localMDPThm}, Lemma \ref{MDPexptightI} and Lemma \ref{MDPexptightII}  provide
actually the moderate deviation principle for $Z_{t}^{\epsilon}$ with respect to the uniform topology on $D[0,T]$, see e.g. Lemma A.1 in \cite{DGWU}, or Theorem 4.14 \cite{FK}.

(ii). For stochastic dynamics driven by Brownian motion, the limit of its standardization is an Ornstein-Uhlenbeck process driven by the same Brownian motion, and so, the fluctuations and the moderate deviations can be established by estimating deviation inequality of the standardization with the Ornstein-Uhlenbeck process (see,e.g. \cite{GaoWang}). That approach cannot be applied to stochastic dynamics with jumps in our paper.
\end{remark}

%%%%%%%%%%%%%%%%%%%%%%%%%%%%%%%%%%%%%%%%%%%%%%%%%%%%%
\section{Asymptotics for the mean process for high-dimensional Hawkes processes}\label{SecMeanProcess}

All the previous results that we derived in Theorem \ref{CLTThm}, Theorem \ref{LDPThm} and Theorem \ref{MDPThm}
for the univariate Hawkes process can be transferred to the mean process of a multivariate Hawkes process.
Consider the $N$-dimensional multivariate Hawkes process using the Poisson embeddings representation:
$(Z^{N,1}_t,\dots,Z^{N,N}_t)_{t\geq 0}$:
\begin{equation} 
Z^{N,i}_t=\int_0^t \int_0^\infty {\bf 1}_{\left\{z \leq \phi \left(N^{-1}\sum_{j=1}^N\int_0^{s-}h(s-u)dZ_u^{N,j}\right)\right\}}
\pi^i(ds\,dz),
\end{equation}
and its mean process $\overline{Z}^{N}_t=\frac{1}{N}\sum_{i=1}^N Z^{N,i}_t$, 
which satisfies
\begin{equation}
\overline{Z}^{N}_t=\int_0^t \int_0^\infty {\bf 1}_{\left\{z \leq \phi \left(\int_0^{s-}h(s-u)d\overline{Z}_u^{N}\right)\right\}}
\frac{1}{N}\sum_{i=1}^{N}\pi^{i}(dsdz).
\end{equation}

Theorem \ref{CLTThm}, Theorem \ref{LDPThm} and Theorem \ref{MDPThm}
for the univariate Hawkes process can be transferred to the following
Theorem \ref{mean-proce-CLTThm}, Theorem \ref{mean-proce-LDPThm} and Theorem \ref{mean-proce-MDPThm} respectively for 
the mean process of a multivariate Hawkes process.

\begin{theorem}\label{mean-proce-CLTThm}
Suppose Assumption \ref{AssumpI}, Assumption \ref{AssumpII} and Assumption \ref{AssumpIV} hold. 
Set $X^N_t:=\sqrt{N}(\overline{Z}_{t}^{N}-m_t)$. Then
$X^N_t$ converges in distribution on $D[0,T]$ to a continuous Gaussian process $X_t$ defined in
Theorem \ref{CLTThm}.
\end{theorem}

\begin{theorem}\label{mean-proce-LDPThm}
Suppose Assumption \ref{AssumpI}, Assumption \ref{AssumpIII}
and Assumption \ref{AssumpV} hold.
$\mathbb{P}(\overline{Z}_{t}^{N}\in\cdot)$ satisfies a large deviation
principle on $D[0,T]$ equipped with Skorokhod topology
with the speed $N$ and the rate function given in Theorem \ref{LDPThm}.
\end{theorem}

\begin{theorem}\label{mean-proce-MDPThm}
Suppose Assumption \ref{AssumpI}, Assumption \ref{AssumpII}, Assumption \ref{AssumpIII},
Assumption \ref{AssumpIV} and Assumption \ref{AssumpV} hold.
Let $a(N)$ be a positive sequence such that $a(N),\frac{N}{a(N)^{2}}\rightarrow \infty$
as $N\rightarrow \infty $.
Then, $\mathbb{P}\left(\frac{\sqrt{N}(Z_{t}^{N}-m_t)}{a(N)}\in\cdot\right)$ satisfies a large deviation
principle on $D[0,T]$ equipped with Skorokhod topology
with speed ${a(N)^{2}}$ and the rate function given in Theorem \ref{MDPThm}.
\end{theorem}

%%%%%%%%%%%%%%%%%%%%%%%%%%%%%%%%%%%%%%%%%%%%%%%%%%%%%%%%%%%%%%%%%%%%

\section{Proofs}\label{ProofsSection}

\subsection{Proofs for Section \ref{CLTSection}}

Before we prove Theorem \ref{CLTThm}, let us
first give the proofs of Lemma \ref{Z1Bound}, Lemma \ref{X2Bound}, and Lemma \ref{Xtight}.

Firstly,  For any $\theta\in\mathbb{R}$,
$e^{\theta N_{T}^{\epsilon}
-\int_{0}^{T}(e^{\theta}-1)\frac{1}{\epsilon}\phi(\int_{0}^{t}\epsilon h(t-s)dN_{s}^{\epsilon})dt}$ 
is a positive local martingale, hence a supermartingale, and thus
for any $\theta>0$, we have
\begin{align}
1&\geq\mathbb{E}\left[e^{\theta N_{T}^{\epsilon}
-\int_{0}^{T}(e^{\theta}-1)\frac{1}{\epsilon}\phi(\int_{0}^{t}\epsilon h(t-s)dN_{s}^{\epsilon})dt}\right]
\\
&\geq\mathbb{E}\left[e^{\theta N_{T}^{\epsilon}
-(e^{\theta}-1)\frac{1}{\epsilon}\phi(0)T
-(e^{\theta}-1)\alpha\int_{0}^{T}\int_{0}^{t}|h(t-s)|dN_{s}^{\epsilon}dt}\right]
\nonumber
\\
&=\mathbb{E}\left[e^{\theta N_{T}^{\epsilon}
-(e^{\theta}-1)\frac{1}{\epsilon}\phi(0)T
-(e^{\theta}-1)\alpha\int_{0}^{T}[\int_{s}^{T}|h(t-s)|dt]dN_{u}^{\epsilon}}\right]
\nonumber
\\
&\geq\mathbb{E}\left[e^{\theta N_{T}^{\epsilon}
-(e^{\theta}-1)\frac{1}{\epsilon}\phi(0)T
-(e^{\theta}-1)\alpha\Vert h\Vert_{L^{1}[0,T]}N_{T}^{\epsilon}}\right].
\nonumber
\end{align}
Since we assumed that $\alpha\Vert h\Vert_{L^{1}[0,T]}<1$, 
for sufficiently small $\theta>0$, we have $\theta-(e^{\theta}-1)\alpha\Vert h\Vert_{L^{1}[0,T]}>0$.
It follows that 
\begin{equation}\label{thetaexp}
\mathbb{E}\left[e^{(\theta-(e^{\theta}-1)\alpha\Vert h\Vert_{L^{1}[0,T]})N_{T}^{\epsilon}}\right]
\leq e^{(e^{\theta}-1)\phi(0)T\frac{1}{\epsilon}}.
\end{equation}
In particular,  for any $k\geq 1$,
\begin{equation}\label{k-order-moment}
\mathbb{E}[(Z_{T}^{\epsilon})^k]<\infty.
\end{equation}

\begin{proof}[Proof of Lemma \ref{Z1Bound}]
Notice that for any $0\leq t\leq T$,
\begin{align}
\mathbb{E}[Z_{t}^{\epsilon}]
&=\mathbb{E}\int_{0}^{t}\phi\left(\int_{0}^{s}h(s-u)dZ_{u}^{\epsilon}\right)ds
\\
&\leq\phi(0)t
+\alpha\mathbb{E}\int_{0}^{T}\int_{0}^{s}|h(s-u)|dZ_{u}^{\epsilon}ds
\nonumber
\\
&\leq\phi(0)t
+\alpha\Vert h\Vert_{L^{\infty}[0,T]}\int_{0}^{T}\mathbb{E}[Z_{s}^{\epsilon}]ds.
\nonumber
\end{align}
The result follows from the Gronwall's inequality.
\end{proof}

%%%%%%%%%%%%%%%%%%%%%%%%%%%%%%%%%%%%%%%%%%%%%%%%%%%%%%%%%%%%%%%%%%%%%%%%%%

\begin{proof}[Proof of Lemma \ref{X2Bound}]
Notice that 
\begin{equation}
X_{t}^{\epsilon}=
\frac{1}{\sqrt{\epsilon}}M_{t}^{\epsilon}
+\frac{1}{\sqrt{\epsilon}}\left[
\int_{0}^{t}\phi\left(\int_{0}^{s}h(s-u)dZ_{u}^{\epsilon}\right)ds
-\int_{0}^{t}\phi\left(\int_{0}^{s}h(s-u)dZ_{u}^{0}\right)ds\right],
\end{equation}
where
\begin{equation}\label{MtEqn}
M_{t}^{\epsilon}:=Z_{t}^{\epsilon}-\int_{0}^{t}\phi\left(\int_{0}^{s}h(s-u)dZ_{u}^{\epsilon}\right)ds
\end{equation}
is a martingale. For any $0\leq t\leq T$,
\begin{align}
|X_{t}^{\epsilon}|
&\leq\frac{1}{\sqrt{\epsilon}}\sup_{0\leq t\leq T}|M_{t}^{\epsilon}|
+\frac{\alpha}{\sqrt{\epsilon}}
\int_{0}^{t}\left|\phi\left(\int_{0}^{s}h(s-u)dZ_{u}^{\epsilon}\right)
-\phi\left(\int_{0}^{s}h(s-u)dZ_{u}^{0}\right)\right|ds
\nonumber
\\
&\leq\frac{1}{\sqrt{\epsilon}}\sup_{0\leq t\leq T}|M_{t}^{\epsilon}|
+\alpha\left[|h(0)|+\int_{0}^{T}|h'(t)|dt\right]
\int_{0}^{t}\sup_{0\leq u\leq s}|X_{u}^{\epsilon}|ds
\nonumber
\end{align}
By Gronwall's inequality,
\begin{equation}\label{X-leq-M-ineq}
\sup_{0\leq t\leq T}
|X_{t}^{\epsilon}|
\leq\frac{1}{\sqrt{\epsilon}}\sup_{0\leq t\leq T}|M_{t}^{\epsilon}|
e^{\alpha\left[|h(0)|+\int_{0}^{T}|h'(t)|dt\right]T}.
\end{equation}
Finally, by Doob's inequality
\begin{align}
\frac{1}{\epsilon}\mathbb{E}\left[\left(\sup_{0\leq t\leq T}|M_{t}^{\epsilon}|\right)^{2}\right]
&\leq\frac{4}{\epsilon}\mathbb{E}\left[(M_{T}^{\epsilon})^{2}\right]
\\
&=4\mathbb{E}\int_{0}^{T}\phi\left(\int_{0}^{s}h(s-u)dZ_{u}^{\epsilon}\right)ds
\nonumber
\\
&=4\mathbb{E}[Z_{T}^{\epsilon}],
\nonumber
\end{align}
where we have proved in Lemma \ref{Z1Bound} that $\mathbb{E}[Z_{T}^{\epsilon}]$ is uniformly bounded in $\epsilon$.
\end{proof}

%%%%%%%%%%%%%%%%%%%%%%%%%%%%%%%%%%%%%%%%%%%%%%%%%%%%%%%%%%%%%

\begin{proof}[Proof of Lemma \ref{Xtight}]
Let us recall that
\begin{equation}
X_{t}^{\epsilon}=
\frac{1}{\sqrt{\epsilon}}M_{t}^{\epsilon}
+\frac{1}{\sqrt{\epsilon}}\left[
\int_{0}^{t}\phi\left(\int_{0}^{s}h(s-u)dZ_{u}^{\epsilon}\right)ds
-\int_{0}^{t}\phi\left(\int_{0}^{s}h(s-u)dZ_{u}^{0}\right)ds\right],
\end{equation}
where $M_{t}^{\epsilon}$ is a martingale
and we can show that for any $\eta>0$, 
\begin{equation}\label{est:martingale}
\lim_{\delta\rightarrow 0}\lim_{\epsilon\rightarrow 0}\mathbb{P}\left(\sup_{0\leq s,t\leq T,|s-t|\leq\delta}\left|\frac{1}{\sqrt{\epsilon}}M_{t}^{\epsilon}-\frac{1}{\sqrt{\epsilon}}M_{s}^{\epsilon}\right|\geq\eta\right)=0.
\end{equation}
To show \eqref{est:martingale}, w.l.o.g., assume $T/\delta\in\mathbb{N}$ and
by using Doob's martingale inequality, Chebychev's inequality, and Burkholder-Davis-Gundy inequality, we have
\begin{align}
&\mathbb{P}\left(\sup_{0\leq s,t\leq T,|s-t|\leq\delta}\left|\frac{1}{\sqrt{\epsilon}}M_{t}^{\epsilon}-\frac{1}{\sqrt{\epsilon}}M_{s}^{\epsilon}\right|\geq\eta\right)
\nonumber
\\
&\leq\sum_{n=1}^{T/\delta}
\mathbb{P}\left(\left|\frac{1}{\sqrt{\epsilon}}M_{n\delta}^{\epsilon}-\frac{1}{\sqrt{\epsilon}}M_{(n-1)\delta}^{\epsilon}\right|
\geq\frac{\eta}{2}\right)
\nonumber
\\
&\leq\frac{C}{\epsilon^{2}\eta^{4}}\sum_{n=1}^{T/\delta}
\mathbb{E}\left[(M_{n\delta}^{\epsilon}-M_{(n-1)\delta}^{\epsilon})^{4}\right]
\nonumber
\\
&=\frac{C'}{\eta^{4}}\sum_{n=1}^{T/\delta}
\mathbb{E}\left[\left(\int_{(n-1)\delta}^{n\delta}\phi\left(\int_{0}^{s-}h(s-u)dZ_{u}^{\epsilon}\right)ds\right)^{2}\right]
\nonumber
\\
&\leq\frac{C'}{\eta^{4}}\delta T\mathbb{E}\left[\left(\sup_{0\leq t\leq T}\phi\left(\int_{0}^{t-}h(t-u)dZ_{u}^{\epsilon}\right)\right)^{2}\right]
\nonumber
\\
&\leq\frac{C'}{\eta^{4}}\delta T\mathbb{E}\left[(\phi(0)+\alpha\Vert h\Vert_{L^{\infty}[0,T]}Z_{T}^{\epsilon})^{2}\right].
\label{gotozero}
\end{align}
Note that for every $0\leq t\leq T$,
\begin{align*}
&\mathbb{E}[(Z_{t}^{\epsilon})^{2}]
\\
&\leq 2\mathbb{E}[(M_{t}^{\epsilon})^{2}]
+2\mathbb{E}\left[\left(\int_{0}^{t}\phi\left(\int_{0}^{s-}h(s-u)dZ_{u}^{\epsilon}\right)ds\right)^{2}\right]
\\
&\leq 2\mathbb{E}[Z_{t}^{\epsilon}]
+2T\mathbb{E}\int_{0}^{t}\phi^{2}\left(\int_{0}^{s-}h(s-u)dZ_{u}^{\epsilon}\right)ds
\\
&\leq 2\mathbb{E}[Z_{t}^{\epsilon}]
+2T\mathbb{E}\int_{0}^{t}(\phi(0)+\alpha\Vert h\Vert_{L^{\infty}[0,T]}Z_{s}^{\epsilon})^{2}ds
\\
&\leq 2\mathbb{E}[Z_{T}^{\epsilon}]
+2T^{2}\phi(0)^{2}+4T^{2}\phi(0)\alpha\Vert h\Vert_{L^{\infty}[0,T]}\mathbb{E}[Z_{T}^{\epsilon}]
+2T(\alpha\Vert h\Vert_{L^{\infty}[0,T]})^{2}\int_{0}^{t}\mathbb{E}[(Z_{s}^{\epsilon})^{2}]ds.
\end{align*}
Recall that we have proved in Lemma \ref{Z1Bound} that $\mathbb{E}[Z_{T}^{\epsilon}]$ is uniformly bounded in $\epsilon$.
By Gronwall's inequality, $\mathbb{E}[(Z_{T}^{\epsilon})^{2}]$ is uniformly bounded in $\epsilon$.
Hence, we conclude that \eqref{est:martingale} follows from \eqref{gotozero} since it goes to zero
as $\delta\rightarrow 0$ uniformly in $\epsilon$.

Moreover, for any $0\leq t\leq t+\delta\leq T$,
\begin{align}
&\bigg|\frac{1}{\sqrt{\epsilon}}
\int_{0}^{t+\delta}\left[\phi\left(\int_{0}^{s}h(s-u)dZ_{u}^{\epsilon}\right)-\phi\left(\int_{0}^{s}h(s-u)dZ_{u}^{0}\right)\right]ds
\\
&\qquad
-\frac{1}{\sqrt{\epsilon}}
\int_{0}^{t}\left[\phi\left(\int_{0}^{s}h(s-u)dZ_{u}^{\epsilon}\right)-\phi\left(\int_{0}^{s}h(s-u)dZ_{u}^{0}\right)\right]ds
\bigg|
\nonumber
\\
&\leq
\frac{1}{\sqrt{\epsilon}}
\int_{t}^{t+\delta}\left|\phi\left(\int_{0}^{s}h(s-u)dZ_{u}^{\epsilon}\right)-\phi\left(\int_{0}^{s}h(s-u)dZ_{u}^{0}\right)\right|ds
\nonumber
\\
&\leq\delta\alpha\left[|h(0)|+\int_{0}^{T}|h'(t)|dt\right]
\sup_{0\leq t\leq T}|X_{t}^{\epsilon}|.
\nonumber
\end{align}
It follows from Lemma \ref{X2Bound} that the sequence
\begin{equation}
\frac{1}{\sqrt{\epsilon}}
\int_{0}^{t}\left[\phi\left(\int_{0}^{s}h(s-u)dZ_{u}^{\epsilon}\right)-\phi\left(\int_{0}^{s}h(s-u)dZ_{u}^{0}\right)\right]ds
\end{equation}
is tight on $C[0,T]$.  Hence,   for any $\eta>0$, 
\begin{equation}
\lim_{\delta\rightarrow 0}\lim_{\epsilon\rightarrow 0}\mathbb{P}\left(\sup_{0\leq s,t\leq T,|s-t|\leq\delta}|X_{t}^{\epsilon}-X_{s}^{\epsilon}|\geq\eta\right)=0,
\end{equation}
which implies that  the sequence $X_{t}^{\epsilon}$ is tight on $D[0,T]$  and the all limits are in $C[0,T]$ by Theorem 15.5 \cite{Billingsley}.
\end{proof}

%%%%%%%%%%%%%%%%%%%%%%%%%%%%%%%%%%%%%%%%%%%%%%%%%%%%%%%%%%%%%%%%%%%%%%%%%%%%%%%%%%%%%%%

We are now finally ready to give the proof of the fluctuations results
in Theorem \ref{CLTThm}.

\begin{proof}[Proof of Theorem \ref{CLTThm}]
We can write 
\begin{align*}
&X_t^\epsilon-\int_{0}^{t}\phi'\left(\int_{0}^{s}h(s-u)dZ_{u}^{0}\right)\left(h(0)X_s^\epsilon+\int_{0}^{s}X_u^\epsilon h'(s-u)du\right)ds\\
&=X_t^\epsilon-\int_{0}^{t}\phi'\left(\int_{0}^{s}h(s-u)dZ_{u}^{0}\right)\int_{0}^{s}h(s-u)dX_{u}^{\epsilon}ds\\
&=\frac{M_t^\epsilon}{\sqrt{\epsilon}} +  \mathcal{E}^{(2)}_{t},
\end{align*}
and
\begin{align*}
&\left(X_t^\epsilon-\int_{0}^{t}\phi'\left(\int_{0}^{s}h(s-u)dZ_{u}^{0}\right)\int_{0}^{s}X_u^{\epsilon}h'(s-u)duds\right)^2-\int_{0}^{t}\phi\left(\int_{0}^{s}h(s-u)dZ_{u}^{\epsilon}\right)ds\\
&=\frac{1}{\epsilon}\left((M_t^\epsilon)^2-\langle M^\epsilon\rangle_t\right)+  \left(\mathcal{E}^{(2)}_{t}\right)^2+2\left(\frac{M_t^\epsilon}{\sqrt{\epsilon}}\right)\mathcal{E}^{(2)}_{t},
\end{align*}
where $M_{t}^{\epsilon}$ is defined in \eqref{MtEqn} and 
\begin{align}
&\mathcal{E}^{(2)}_{t}:=
\frac{1}{\sqrt{\epsilon} }\bigg[
\int_{0}^{t}\phi\left(\int_{0}^{s}h(s-u)dZ_{u}^{\epsilon}\right)ds
-\int_{0}^{t}\phi\left(\int_{0}^{s}h(s-u)dZ_{u}^{0}\right)ds\bigg]
\nonumber
\\
&\qquad
-\frac{1}{\sqrt{\epsilon} }
\int_{0}^{t}\phi'\left(\int_{0}^{s}h(s-u)dZ_{u}^{0}\right)
\left[\int_{0}^{s}h(s-u)dZ_{u}^{\epsilon}
-\int_{0}^{s}h(s-u)dZ_{u}^{0}\right]ds.
\end{align}
Then, by Doob's martingale inequality,  and Burkholder-Davis-Gundy inequality,  there exists a constant $0<C_T<\infty$
\begin{align}
\frac{1}{\epsilon^2}\mathbb{E}\left[\left(\sup_{0\leq t\leq T}|M_{t}^{\epsilon}|\right)^{4}\right]
&\leq C_T \mathbb{E}\int_{0}^{T}\left(\phi\left(\int_{0}^{s}h(s-u)dZ_{u}^{\epsilon}\right)\right)^2ds
\nonumber
\\
&\leq C_T T\mathbb{E}\left[(\phi(0)+\alpha\Vert h\Vert_{L^{\infty}[0,T]}Z_{T}^{\epsilon})^{2}\right],
\nonumber
\end{align}
where we have proved in Lemma \ref{X2Bound} that $\mathbb{E}[(Z_{T}^{\epsilon})^2]$ is uniformly bounded in $\epsilon$.   Thus,  $\frac{1}{\epsilon^2}\mathbb{E}\left[\left(\sup_{0\leq t\leq T}|M_{t}^{\epsilon}|\right)^{4}\right]$ is uniformly bounded in $\epsilon$, and by
\eqref{X-leq-M-ineq},   $\frac{1}{\epsilon^2}\mathbb{E}\left[\left(\sup_{0\leq t\leq T}|X_{t}^{\epsilon}|\right)^{4}\right]$ is also uniformly bounded in $\epsilon$.

By Taylor expansion,
\begin{align}
\sup_{0\leq t\leq T}|\mathcal{E}_{t}^{(2)}|
&\leq
\frac{1}{\sqrt{\epsilon} }
\Vert\phi''\Vert\int_{0}^{T}
\left[\int_{0}^{s}h(s-u)dZ_{u}^{\epsilon}
-\int_{0}^{s}h(s-u)dZ_{u}^{0}\right]^{2}ds
\\
&\leq
\sqrt{\epsilon} T\Vert\phi''\Vert
\alpha^{2}\left[|h(0)|+\int_{0}^{T}|h'(t)|dt\right]^{2}
\sup_{0\leq t\leq T}(X_{t}^{\epsilon})^{2}.
\nonumber
\end{align}

Therefore,   we have that
\begin{equation*}
\frac{M_t^\epsilon}{\sqrt{\epsilon}} , \quad \frac{1}{\epsilon}\left((M_t^\epsilon)^2-\langle M^\epsilon\rangle_t\right), \quad 0<\epsilon\leq\epsilon_0 
\end{equation*}
are uniformly integrable martingales, and
\begin{equation}
\mathbb{E}\left[\sup_{0\leq t\leq T}\left|\mathcal{E}_{t}^{(2)}\right|^2\right]
=O(\epsilon).
\end{equation}
These yield that as $\epsilon\to 0$, in probability, 
\begin{equation*}
\sup_{0\leq t\leq T}\left|\mathcal{E}_{t}^{(2)}\right|\to 0,\quad \sup_{0\leq t\leq T}\left|\left(\frac{M_t^\epsilon}{\sqrt{\epsilon}}\right)\mathcal{E}^{(2)}_{t}\right|\to 0,
\end{equation*}
and there exists a square integrable martingale $\tilde{M}_t,t\in [0,T]$ such that
\begin{equation*}
\sup_{0\leq t\leq T}\left|\frac{M_t^\epsilon}{\sqrt{\epsilon}}-\tilde{M}_t\right|\to 0\quad \mbox{ and }\quad  \sup_{0\leq t\leq T}\left|\frac{1}{\epsilon}\left((M_t^\epsilon)^2-\langle M^\epsilon\rangle_t\right)-\left((\tilde{M}_t)^2-\langle \tilde{M}\rangle_t\right)\right|\to 0.
\end{equation*}
  
In Lemma \ref{Xtight}, we showed that the sequence $X_{t}^{\epsilon}$ is tight on $D[0,T]$.
Let $X$ be a limit point of $X^\epsilon$, and $X$ is continuous in $t$.  
We conclude that,
\begin{equation*}
M_t:=  X_t -   \int_{0}^{t}\phi'\left(\int_{0}^{s}h(s-u)dZ_{u}^{0}\right)\int_{0}^{s}X_u h'(s-u)duds,
\end{equation*}
and
\begin{equation*}
\begin{aligned}
N_t:=&\left(X_t- \int_{0}^{t}\phi'\left(\int_{0}^{s}h(s-u)dZ_{u}^{0}\right)\left(h(0)X_s+\int_{0}^{s}X_uh'(s-u)du\right)ds\right)^2\\
&\qquad\qquad
- \int_{0}^{t}\phi\left(\int_{0}^{s}h(s-u)dZ_{u}^{0}\right)ds,
\end{aligned}
\end{equation*}
are martingales.   
Since $X_{t}$ is continuous in time $t$, by L\'{e}vy's characterization of Brownian motion
and martingale representation theorem, 
see e.g. Chapter IV Theorem 3.6. and Chapter V Proposition 3.8. \cite{RevuzYor},
there exists a standard Brownian motion $W_t$, such that \eqref{Gauss-process} holds.

By Gronwall's inequality,  the stochastic differential equation \eqref{Gauss-process} only has a unique solution which  implies that as $\epsilon\to 0$, the set of limit points of  $\{X^\epsilon\}$ is a singleton.  Thus,  $X^\epsilon$ converges in distribution on $D[0,T]$ to the solution of the equation \eqref{Gauss-process}.

Finally, let us show that the limit $X_{t}$ is a Gaussian process.
Set $X_{t}^{(0)}:=0$ and
\begin{equation}
X_{t}^{(1)}:=\int_{0}^{t}\sqrt{\phi\left(\int_{0}^{s}h(s-u)dZ_{u}^{0}\right)}dW_{s},
\end{equation} 
and for every $n\geq 1$,
\begin{equation}
\begin{aligned}
X_{t}^{(n+1)}:=&\int_{0}^{t}\phi'\left(\int_{0}^{s}h(s-u)dZ_{u}^{0}\right)\left(h(0)X_{s}^{(n)}+\int_{0}^{s}X_{u}^{(n)}h'(s-u)du\right)ds \\
&+ \int_{0}^{t}\sqrt{ \phi\left(\int_{0}^{s}h(s-u)dZ_{u}^{0}\right)}dW_{s},
\end{aligned}
\end{equation}
Then $\{X_{t}^{(n)},t\in[0,T]\}_{n\geq 1}$ is a sequence of Gaussian processes.
Moreover, we can compute that
\begin{align*}
&X_{t}^{(n+1)}-X_{t}^{(n)}
\\
&=\int_{0}^{t}\phi'\left(\int_{0}^{s}h(s-u)dZ_{u}^{0}\right)
\left[h(0)(X_{s}^{(n)}-X_{s}^{(n-1)})+\int_{0}^{s}(X_{u}^{(n)}-X_{u}^{(n-1)})h'(s-u)du\right]ds,
\end{align*}
where we used the integration by parts and $X_{0}^{(n)}=X_{0}^{(n-1)}=0$.
Set $\Phi^{(n)}(t):=\sup_{0\leq s\leq t}|X_{s}^{(n)}-X_{s}^{(n-1)}|$. 
Then for every $t\in[0,T]$,
\begin{align*}
\Phi^{(n+1)}(t)&\leq\alpha\int_{0}^{t}\left(|h(0)|+\int_{0}^{s}|h'(s-u)|du\right)\Phi^{(n)}(s)ds
\\
&\leq\alpha(|h(0)|+\Vert h'\Vert_{L^{1}[0,T]})\int_{0}^{t}\Phi^{(n)}(s)ds,
\end{align*}
which implies that
\begin{equation}
\Phi^{(n+1)}(T)\leq\frac{(\alpha(|h(0)|+\Vert h'\Vert_{L^{1}[0,T]})T)^{n}}{n!}
\sup_{0\leq t\leq T}\Phi^{(1)}(t),
\end{equation}
which yields that
\begin{equation}
\mathbb{E}\left[\sum_{n=1}^{\infty}(\Phi^{(n)}(T))\right]<\infty.
\end{equation}
Thus, almost surely, $\sum_{n=1}^{\infty}\Phi^{(n)}(T)<\infty$.
Thus, by Proposition 6.1 (Chapter 0) in \cite{RevuzYor}, 
$\tilde{X}_{t}=\sum_{n=0}^{\infty}(X_{t}^{(n+1)}-X_{t}^{(n)})$ is a continuous Gaussian process
such that 
$$
\sup_{0\leq t\leq T}|X_{t}^{(n)}-\tilde{X}_{t}|\rightarrow 0 ~~\mbox{ almost surely}.  
$$
Therefore,
\begin{equation}
\begin{aligned}
\tilde{X}_t=&\int_{0}^{t}\phi'\left(\int_{0}^{s}h(s-u)dZ_{u}^{0}\right)\left(h(0)\tilde{X}_s+\int_{0}^{s}\tilde{X}_{u}h'(s-u)du\right)ds\\
& + \int_{0}^{t}\sqrt{ \phi\left(\int_{0}^{s}h(s-u)dZ_{u}^{0}\right)}dW_{s}.
\end{aligned}
\end{equation}
By the uniqueness of the solution of the equation  \eqref{Gauss-process}, we have $\tilde{X}=X$. Therefore, $\{X_t,t\in[0,T]\}$ is a Gaussian process.
\end{proof}

%%%%%%%%%%%%%%%%%%%%%%%%%%%%%%%%%%%%%%%%%%%%%%%%%%%%%%%%%%%%%%%%%%%%

\subsection{Proofs for Section \ref{LDPSection}}

\begin{proof}[Proof of Theorem \ref{LDPThm}]
Theorem \ref{LDPThm} follows from  
the local large deviation principle in Theorem \ref{localLDPThm}
and the super-exponential estimates in Lemma \ref{exptightI} and Lemma \ref{exptightII}.
\end{proof}

\begin{proof}[Proof of Theorem \ref{localLDPThm}]
Set
$$
\mathcal M_0[0,T]=\left\{\eta\in D[0,T]; ~\eta(0)=0, ~\eta(t) \mbox{  is  non-decreasing in } t\in [0,T]\right\}.
$$
Then $\mathcal M_0[0,T]$ is a closed subset in $D[0,T]$ 
and $\mathbb{P}(Z^\epsilon \in \mathcal M_0[0,T] \mbox{ for all } \epsilon \in (0,1])=1$, 
Thus,  for any $\eta\not\in \mathcal M_0[0,T]$,  
$$
\lim_{\delta\rightarrow 0}\lim_{\epsilon\rightarrow 0}
\epsilon\log
\mathbb{P}\left(\sup_{0\leq t\leq T}|Z_{t}^{\epsilon}-\eta(t)|\leq\delta\right)
=-\infty.
$$
Next, we assume that  $\eta\in \mathcal M_0[0,T]$.  

Let $\tilde{\mathbb{P}}$ be the probability measure
under which $N^{\epsilon}$ is a standard Poisson process with
intensity $\frac{1}{\epsilon}$. 
Since $\phi$ is $\alpha$-Lipschitz, we have
\begin{equation}
\frac{1}{\epsilon}\phi\left(\int_{0}^{t-}\epsilon h(t-s)dN_{s}^{\epsilon}\right)
\leq\frac{1}{\epsilon}\phi(0)
+\frac{\alpha}{\epsilon}\int_{0}^{t-}\epsilon|h(t-s)|dN_{s}^{\epsilon}
\leq\frac{1}{\epsilon}\phi(0)
+\alpha\Vert h\Vert_{L^{\infty}[0,T]}N_{t-}.
\end{equation}
That is, the intensity has at most the linear growth in $N_{t-}$.
Moreover, under our assumption, we have $\inf_{x\geq 0}\phi(x)>0$. Thus, $\mathbb{P}$ and $\tilde{\mathbb{P}}$
are equivalent, and the Radon-Nikodym is given by, see e.g. \cite{SH},
\begin{equation}
\frac{d\mathbb{P}}{d\tilde{\mathbb{P}}}\bigg|_{\mathcal{F}_{T}}
=e^{\int_{0}^{T}\log\left(\frac{\frac{1}{\epsilon}\phi\left(\int_{0}^{t-}\epsilon h(t-s)dN_{s}^{\epsilon}\right)}{\frac{1}{\epsilon}}\right)dN_{t}^{\epsilon}
-\int_{0}^{T}\left[\frac{1}{\epsilon}\phi\left(\int_{0}^{t}\epsilon h(t-s)dN_{s}^{\epsilon}\right)-\frac{1}{\epsilon}\right]ds}.
\end{equation}
By changing of the probability measure $\mathbb{P}$
to $\tilde{\mathbb{P}}$, 
\begin{align}
&\mathbb{P}\left(\sup_{0\leq t\leq T}|Z_{t}^{\epsilon}-\eta(t)|\leq\delta\right)
\\
&=\tilde{\mathbb{E}}\left[e^{\int_{0}^{T}\log\left(\frac{\frac{1}{\epsilon}\phi\left(\int_{0}^{t-}\epsilon h(t-s)dN_{s}^{\epsilon}\right)}{\frac{1}{\epsilon}}\right)dN_{t}^{\epsilon}
-\int_{0}^{T}\left[\frac{1}{\epsilon}\phi\left(\int_{0}^{t}\epsilon h(t-s)dN_{s}^{\epsilon}\right)-\frac{1}{\epsilon}\right]ds}
1_{\sup_{0\leq t\leq T}|Z_{t}^{\epsilon}-\eta(t)|\leq\delta}\right]
\nonumber
\\
&=\tilde{\mathbb{E}}\left[e^{\frac{1}{\epsilon}\int_{0}^{T}\log\left(\phi\left(\int_{0}^{t-} h(t-s)dZ_{s}^{\epsilon}\right)\right)dZ_{t}^{\epsilon}
-\frac{1}{\epsilon}
\int_{0}^{T}\left[\phi\left(\int_{0}^{t}h(t-s)dZ_{s}^{\epsilon}\right)-1\right]ds}
1_{\sup_{0\leq t\leq T}|Z_{t}^{\epsilon}-\eta(t)|\leq\delta}\right].
\nonumber
\end{align}
For any $\{Z^{\epsilon}_{t}, 0\leq t\leq T\}$ with $\sup_{0\leq t\leq T}|Z_{t}^{\epsilon}-\eta(t)|\leq\delta$, we have
\begin{align}
&\left|\int_{0}^{T}\log\left(\phi\left(\int_{0}^{t-} h(t-s)dZ_{s}^{\epsilon}\right)\right)dZ_{t}^{\epsilon}
-\int_{0}^{T}\log\left(\phi\left(\int_{0}^{t-} h(t-s)d\eta(s)\right)\right)d\eta(t)\right|
\\
&\leq\left|\int_{0}^{T}\log\left(\phi\left(\int_{0}^{t-} h(t-s)d\eta(s)\right)\right)dZ_{t}^{\epsilon}
-\int_{0}^{T}\log\left(\phi\left(\int_{0}^{t-} h(t-s)d\eta(s)\right)\right)d\eta(t)\right|
\nonumber
\\
&\qquad
+\left|\int_{0}^{T}\log\left(\phi\left(\int_{0}^{t-} h(t-s)dZ_{s}^{\epsilon}\right)\right)dZ_{t}^{\epsilon}
-\int_{0}^{T}\log\left(\phi\left(\int_{0}^{t-} h(t-s)d\eta(s)\right)\right)dZ_{t}^{\epsilon}\right|.
\nonumber
\end{align}

Set $\nu_{\eta}(t)=\phi\left(\int_{0}^{t-} h(t-s)d\eta(s)\right)$. 
Then $\nu_\eta$  is a finite variation function on $[0,T]$. The total variation of $\nu_{\eta}(t)$ on $[0,T]$ denotes by 
$\int_0^T|d\nu_{\eta}(t)|$. 
Then
$$
|\log \nu_{\eta}(T)|\leq |\log\phi(0)|
+\frac{\alpha}{\inf_{x\geq 0}\phi(x)}\Vert h\Vert_{L^{\infty}[0,T]}\eta(T),
$$
and
$$
\int_0^T|d\nu_{\eta}(T)|\leq \alpha (|h(0)|+\Vert h'\Vert_{L^{\infty}[0,T]})(\eta(T)+T\eta(T)).
$$

It follows from integration by parts that
\begin{align}
&\left|\int_{0}^{T}\log\left(\phi\left(\int_{0}^{t-} h(t-s)d\eta(s)\right)\right)dZ_{t}^{\epsilon}
-\int_{0}^{T}\log\left(\phi\left(\int_{0}^{t-} h(t-s)d\eta(s)\right)\right)d\eta(t)\right|
\nonumber\\
&\leq
\left|\int_{0}^{T}\frac{Z_{t}^{\epsilon}-\eta(t)}
{\nu_\eta(t)}d\nu_\eta(t)
 \right|
+\left|\log\left(\nu_\eta(T)\right)Z_{T}^{\epsilon}
-\log\left(\nu_\eta(T)\right)\eta(T)\right|
\nonumber
\\
&\leq
\sup_{0\leq t\leq T}|Z_{t}^{\epsilon}-\eta(t)|\frac{1}{\inf_{x\geq 0}\phi(x)}
\int_{0}^{T}\left|d\nu_\eta(t)\right|
+\sup_{0\leq t\leq T}|Z_{t}^{\epsilon}-\eta(t)|\left|\log\left(\nu_\eta(T)\right)\right|
\nonumber
\\
&\leq
\sup_{0\leq t\leq T}|Z_{t}^{\epsilon}-\eta(t)|\frac{\alpha}{\inf_{x\geq 0}\phi(x)}
(|h(0)|+\Vert h'\Vert_{L^{\infty}[0,T]})(\eta(T)+T\eta(T))
\nonumber
\\
&\qquad
+\sup_{0\leq t\leq T}|Z_{t}^{\epsilon}-\eta(t)|
\left[|\log\phi(0)|
+\frac{\alpha}{\inf_{x\geq 0}\phi(x)}\Vert h\Vert_{L^{\infty}[0,T]}\eta(T)\right].
\nonumber
\end{align}
On the other hand, we can estimate that
\begin{align}
&\left|\int_{0}^{T}\log\left(\phi\left(\int_{0}^{t-} h(t-s)dZ_{s}^{\epsilon}\right)\right)dZ_{t}^{\epsilon}
-\int_{0}^{T}\log\left(\phi\left(\int_{0}^{t} h(t-s)d\eta(s)\right)\right)dZ_{t}^{\epsilon}\right|
\\
&\leq\sup_{0\leq t\leq T}\left|\log\left(\phi\left(\int_{0}^{t} h(t-s)dZ_{s}^{\epsilon}\right)\right)
-\log\left(\phi\left(\int_{0}^{t} h(t-s)d\eta(s)\right)\right)\right|
Z_{T}^{\epsilon}
\nonumber
\\
&\leq\frac{\alpha}{\inf_{x\geq 0}\phi(x)}\sup_{0\leq t\leq T}\left|\int_{0}^{t} h(t-s)dZ_{s}^{\epsilon}
-\int_{0}^{t}h(t-s)d\eta(s)\right|
[\eta(T)+\delta]
\nonumber
\\
&\leq\frac{\alpha}{\inf_{x\geq 0}\phi(x)}\left[|h(0)|+\int_{0}^{T}|h'(t)|dt\right]
\sup_{0\leq t\leq T}|Z_{t}^{\epsilon}-\eta(t)|
[\eta(T)+\delta].
\nonumber
\end{align}
Finally, we can estimate that
\begin{align}
&\left|\int_{0}^{T}\left[\phi\left(\int_{0}^{t}h(t-s)dZ_{s}^{\epsilon}\right)-1\right]ds
-\int_{0}^{T}\left[\phi\left(\int_{0}^{t}h(t-s)d\eta(s)\right)-1\right]ds\right|
\\
&\leq
T\alpha\left[|h(0)|+\int_{0}^{T}|h'(t)|dt\right]
\sup_{0\leq t\leq T}|Z_{t}^{\epsilon}-\eta(t)|.
\nonumber
\end{align}
Since $N^{\epsilon}$ is a standard Poisson process with
intensity $\frac{1}{\epsilon}$ under the probability measure $\tilde{\mathbb{P}}$,
it is well known that, see, e.g. \cite{LS,Mog}, $\tilde{\mathbb{P}}(Z_{t}^{\epsilon}\in\cdot)$ satisfies
a large deviation principle on $D[0,T]$ with the rate function
\begin{equation}
I_{Pos}(\eta)=
\begin{cases}
\int_{0}^{T}\ell(\eta'(t);1)dt  
&\text{if $\eta\in\mathcal{AC}_{0}^{+}[0,T]$}
\\
+\infty &\text{otherwise}
\end{cases}, 
\end{equation}
where $\ell(\cdot)$ is defined in \eqref{ellEqn}.

Hence, we conclude that
\begin{align}
&\lim_{\delta\rightarrow 0}\lim_{\epsilon\rightarrow 0}\epsilon
\log\tilde{\mathbb{E}}\left[e^{\int_{0}^{T}\log\left(\frac{\frac{1}{\epsilon}\phi\left(\int_{0}^{t-}\epsilon h(t-s)dN_{s}^{\epsilon}\right)}{\frac{1}{\epsilon}}\right)dN_{t}^{\epsilon}
-\int_{0}^{T}\left[\frac{1}{\epsilon}\phi\left(\int_{0}^{t}\epsilon h(t-s)dN_{s}^{\epsilon}\right)-\frac{1}{\epsilon}\right]ds}
1_{\sup_{0\leq t\leq T}|Z_{t}^{\epsilon}-\eta(t)|\leq\delta}\right]
\\
&=\int_{0}^{T}\log\left(\phi\left(\int_{0}^{t}h(t-s)d\eta(s)\right)\right)d\eta(t)
-\int_{0}^{T}\left[\phi\left(\int_{0}^{t}h(t-s)d\eta(s)\right)-1\right]ds
\nonumber
\\
&\qquad
+
\lim_{\epsilon\rightarrow 0}\epsilon
\log\tilde{\mathbb{P}}\left(\sup_{0\leq t\leq T}|Z_{t}^{\epsilon}-\eta(t)|\leq\delta\right)
\nonumber
\\
&=\int_{0}^{T}\log\left(\phi\left(\int_{0}^{t}h(t-s)d\eta(s)\right)\right)d\eta(t)
-\int_{0}^{T}\left[\phi\left(\int_{0}^{t}h(t-s)d\eta(s)\right)-1\right]ds
-I_{Pos}(\eta)\nonumber
\\
&=-I(\eta).
\nonumber
\end{align}
\end{proof}
%%%%%%%%%%%%%%%%%%%%%%%%%%%%%%%%%%%%%%%%%%%%%%%%%%%%%%%%%%%%%%%%%%%%%%%

\begin{proof}[Proof of Lemma \ref{exptightI}]

By \eqref{thetaexp}, and Chebychev's inequality, for sufficiently small $\theta>0$,
\begin{align}
\mathbb{P}(Z_{T}^{\epsilon}\geq K)
&\leq\mathbb{E}\left[e^{(\theta-(e^{\theta}-1)\alpha\Vert h\Vert_{L^{1}[0,T]})N_{T}^{\epsilon}}\right]
e^{-(\theta-(e^{\theta}-1)\alpha\Vert h\Vert_{L^{1}[0,T]})\frac{K}{\epsilon}}
\\
&\leq
e^{(e^{\theta}-1)\phi(0)T\frac{1}{\epsilon}}
e^{-(\theta-(e^{\theta}-1)\alpha\Vert h\Vert_{L^{1}[0,T]})\frac{K}{\epsilon}},
\nonumber
\end{align}
which yields the desired result.
\end{proof}

%%%%%%%%%%%%%%%%%%%%%%%%%%%%%%%%%%%%%%%%%%%%%%%%%%%%%%%%%%%%%%%%%%%%%

\begin{proof}[Proof of Lemma \ref{exptightII}]
Without loss of generality, let us assume that $MT\in\mathbb{N}$.
For any $\delta>0$,
\begin{align} 
\mathbb{P}\left(\sup_{0\leq s\leq t\leq T,
|t-s|\leq\frac{1}{M}}|Z_{t}^{\epsilon}-Z_{s}^{\epsilon}|\geq\delta\right)
&\leq
\mathbb{P}\left(\exists j, 1\leq j\leq MT:
Z_{\frac{j}{M}}^{\epsilon}-Z_{\frac{j-1}{M}}^{\epsilon}\geq\frac{\delta}{2}\right)
\\
&\leq
\sum_{j=1}^{MT}\mathbb{P}\left(
N_{\frac{j}{M}}^{\epsilon}-N_{\frac{j-1}{M}}^{\epsilon}\geq\frac{\delta}{2\epsilon}\right).
\nonumber
\end{align}
For any $\theta>0$,
\begin{align}
1&=\mathbb{E}\left[e^{\theta(N^{\epsilon}_{j/M}-N^{\epsilon}_{(j-1)/M})
-(e^{\theta}-1)\int_{(j-1)/M}^{j/M}\frac{1}{\epsilon}\phi(\int_{0}^{t}\epsilon h(t-s)dN_{s}^{\epsilon})dt}\right]
\\
&\geq e^{-(e^{\theta}-1)\phi(0)\frac{1}{\epsilon}\frac{1}{M}}
\mathbb{E}\left[e^{\theta(N^{\epsilon}_{j/M}-N^{\epsilon}_{(j-1)/M})
-(e^{\theta}-1)\alpha\int_{(j-1)/M}^{j/M}\int_{0}^{t}h(t-s)dN_{s}^{\epsilon}dt}\right]
\nonumber
\\
&\geq e^{-(e^{\theta}-1)\phi(0)\frac{1}{\epsilon}\frac{1}{M}}
\mathbb{E}\left[e^{\theta(N^{\epsilon}_{j/M}-N^{\epsilon}_{(j-1)/M})
-(e^{\theta}-1)\alpha\Vert h\Vert_{L^{\infty}[0,T]}N_{T}^{\epsilon}\frac{1}{M}}\right].
\nonumber
\end{align}
Therefore, by Cauchy-Schwarz inequality,
\begin{align}
&\mathbb{E}\left[e^{\frac{\theta}{2}(N^{\epsilon}_{j/M}-N^{\epsilon}_{(j-1)/M})}\right]
\\
&=\mathbb{E}\left[e^{\frac{\theta}{2}(N^{\epsilon}_{j/M}-N^{\epsilon}_{(j-1)/M})-\frac{1}{2}(e^{\theta}-1)\alpha\Vert h\Vert_{L^{\infty}[0,T]}N_{T}^{\epsilon}\frac{1}{M}}
e^{\frac{1}{2}(e^{\theta}-1)\alpha\Vert h\Vert_{L^{\infty}[0,T]}N_{T}^{\epsilon}\frac{1}{M}}\right]
\nonumber
\\
&\leq\left(\mathbb{E}\left[e^{\theta(N^{\epsilon}_{j/M}-N^{\epsilon}_{(j-1)/M})-(e^{\theta}-1)\alpha\Vert h\Vert_{L^{\infty}[0,T]}N_{T}^{\epsilon}\frac{1}{M}}\right]\right)^{\frac{1}{2}}
\left(\mathbb{E}\left[e^{(e^{\theta}-1)\alpha\Vert h\Vert_{L^{\infty}[0,T]}N_{T}^{\epsilon}\frac{1}{M}}\right]\right)^{\frac{1}{2}}
\nonumber
\\
&\leq e^{\frac{1}{2}(e^{\theta}-1)\phi(0)\frac{1}{\epsilon}\frac{1}{M}}
\left(\mathbb{E}\left[e^{(e^{\theta}-1)\alpha\Vert h\Vert_{L^{\infty}[0,T]}N_{T}^{\epsilon}\frac{1}{M}}\right]\right)^{\frac{1}{2}},
\nonumber
\end{align}
which is uniform in $1\leq j\leq TM$. By Chebychev's inequality,
\begin{align} 
&\mathbb{P}\left(\sup_{0\leq s\leq t\leq T,
|t-s|\leq\frac{1}{M}}|Z_{t}^{\epsilon}-Z_{s}^{\epsilon}|\geq\delta\right)
\\
&\leq
\sum_{j=1}^{MT}\mathbb{P}\left(
N_{\frac{j}{M}}^{\epsilon}-N_{\frac{j-1}{M}}^{\epsilon}\geq\frac{\delta}{2\epsilon}\right)
\nonumber
\\
&\leq
\sum_{j=1}^{MT}
\mathbb{E}\left[e^{\frac{\theta}{2}(N^{\epsilon}_{j/M}-N^{\epsilon}_{(j-1)/M})}\right]
e^{-\frac{\theta}{2}\frac{\delta}{2\epsilon}}
\nonumber
\\
&\leq
MT e^{\frac{1}{2}(e^{\theta}-1)\phi(0)\frac{1}{\epsilon}\frac{1}{M}}
\left(\mathbb{E}\left[e^{(e^{\theta}-1)\alpha\Vert h\Vert_{L^{\infty}[0,T]}\frac{1}{M}N_{T}^{\epsilon}}\right]\right)^{\frac{1}{2}}
e^{-\frac{\theta}{2}\frac{\delta}{2\epsilon}}.
\nonumber
\end{align}
It follows from \eqref{thetaexp} that
for any sufficiently small $\iota>0$,
\begin{equation}\label{smalltheta}
\mathbb{E}[e^{\iota N_{T}^{\epsilon}}]
\leq e^{C(\iota)\frac{1}{\epsilon}},
\end{equation}
where $C(\iota)$ is a positive constant that depends only on $\iota$, $\alpha$, $\Vert h\Vert_{L^{1}[0,T]}$, 
$\phi(0)$ and $T$.

Let $\gamma$ be a sufficiently small fixed constant, independent of all the other parameters.
We define
$\theta:=\log(1+\gamma M)$,
and thus
\begin{equation}
(e^{\theta}-1)\alpha\Vert h\Vert_{L^{\infty}[0,T]}\frac{1}{M}
=\gamma
\alpha\Vert h\Vert_{L^{\infty}[0,T]}
\end{equation}
is sufficiently small since $\gamma$ is sufficiently small and we can apply \eqref{smalltheta} and the Chebychev's inequality
and get
\begin{align}
&\limsup_{\epsilon\rightarrow 0}\epsilon\log\mathbb{P}\left(\sup_{0\leq s\leq t\leq T,
|t-s|\leq\frac{1}{M}}|Z_{t}^{\epsilon}-Z_{s}^{\epsilon}|\geq\delta\right)
\\
&\leq
\frac{1}{2}(e^{\theta}-1)\phi(0)\frac{1}{M}
+\frac{1}{2}C\left((e^{\theta}-1)\alpha\Vert h\Vert_{L^{\infty}[0,T]}\frac{1}{M}\right)
-\frac{1}{4}\theta\delta.
\nonumber
\end{align}
Since $\theta=\log(1+\gamma M)$, we get
\begin{align}
&\limsup_{\epsilon\rightarrow 0}\epsilon\log\mathbb{P}\left(\sup_{0\leq s\leq t\leq T,
|t-s|\leq\frac{1}{M}}|Z_{t}^{\epsilon}-Z_{s}^{\epsilon}|\geq\delta\right)
\\
&\leq
\frac{1}{2}\phi(0)\gamma
+\frac{1}{2}C\left(\gamma\alpha\Vert h\Vert_{L^{\infty}[0,T]}\right)
-\frac{1}{4}\log(1+\gamma M)\delta,
\nonumber
\end{align}
which yields the desired result by letting $M\rightarrow\infty$.
\end{proof}

%%%%%%%%%%%%%%%%%%%%%%%%%%%%%%%%%%%%%%%%%%%%%%%%%%%%%%%%%%%%%%%%%%%

\subsection{Proofs for Section \ref{MDPSection}}

\begin{proof}[Proof of Theorem \ref{MDPThm}]
Theorem \ref{MDPThm} follows from  
the local large deviation principle in Theorem \ref{localMDPThm}
and the super-exponential estimates in Lemma \ref{MDPexptightI} and Lemma \ref{MDPexptightII}.
\end{proof}

\begin{proof}[Proof of Theorem \ref{localMDPThm}]
Set
$$
\mathcal V_0[0,T]=\left\{\eta\in D[0,T]; ~\eta(0)=0, ~\eta(t) \mbox{  has  finite variation  in } t\in [0,T]\right\}.
$$
Then $\mathcal V_0[0,T]$ is a closed subset in $D[0,T]$ 
and $\mathbb{P}(Z^\epsilon-Z^0 \in \mathcal V_0[0,T] \mbox{ for all } \epsilon\in(0,1])=1$, 
Thus,  for any $\eta\not\in \mathcal V_0[0,T]$,  
$$
\lim_{\delta \rightarrow 0}\lim_{\epsilon\rightarrow 0} \frac{\epsilon}{a^2(\epsilon)}\log
\mathbb{P}\left(\sup_{0\leq t\leq T}\left|\frac{Z_{t}^{\epsilon}-Z_t^0}{a(\epsilon)}-\eta(t)\right|\leq\delta\right)
=-\infty.
$$
Next, we assume that  $\eta\in \mathcal V_0[0,T]$.  

Let $\tilde{\mathbb{P}}$ be the probability measure
under which $N^{\epsilon}$ is an inhomogeneous Poisson process with
intensity $\frac{1}{\epsilon}\phi\left(\int_{0}^{t}  h(t-s)dZ_{s}^0\right)$.
Denote by
$$
  {\eta}^\epsilon(t):= Z_t^0+a(\epsilon)\eta(t),~~~~~\delta(\epsilon):=\delta a(\epsilon).
$$

By changing the probability measure $\mathbb{P}$
to $\tilde{\mathbb{P}}$ (see the discussions about the change of measure
and the Radon-Nikodym derivative in the proof of Theorem \ref{localLDPThm} and \cite{SH}),
\begin{align*}
&\mathbb{P}\left(\sup_{0\leq t\leq T}\left|\frac{Z_{t}^{\epsilon}-Z_t^0}{a(\epsilon)}-\eta(t)\right|\leq\delta\right)
\\
&=\tilde{\mathbb{E}}\Bigg[e^{ \frac{1}{\epsilon} \int_{0}^{T}\log\left( \frac{\phi\left(\int_{0}^{t-}  h(t-s)dZ_{s}^{\epsilon}\right)}{\phi\left(\int_{0}^{t} h(t-s)dZ_s^0\right)}\right)d Z_{t}^{\epsilon}
 -\frac{1}{\epsilon} \int_{0}^{T} \left(\phi\left(\int_{0}^{t}\  h(t-s)dZ_{s}^{\epsilon}\right)-\phi\left(\int_{0}^{t}  h(t-s)dZ_s^0\right)\right) ds}
 \\
 &\qquad\qquad\qquad\cdot
 1_{\sup_{0\leq t\leq T}|Z_t^{\epsilon}-\eta^{\epsilon}(t)|\leq\delta(\epsilon)}\Bigg].
\end{align*}

Replacing  $\eta $ in  the proof of Theorem \ref{localLDPThm} by $\eta^\epsilon$, we have firstly,
\begin{align}\label{estimate1}
&\left|\int_{0}^{T}\log\left(\phi\left(\int_{0}^{t} h(t-s)d\eta^\epsilon(s)\right)\right)dZ_{t}^{\epsilon}
-\int_{0}^{T}\log\left(\phi\left(\int_{0}^{t} h(t-s)d\eta^\epsilon(s)\right)\right)d\eta^\epsilon(t)\right|
\\
&\leq
\sup_{0\leq t\leq T}|Z_{t}^{\epsilon}-\eta^\epsilon(t)|\frac{\alpha}{\inf_{x\geq 0}\phi(x)}
(|h(0)|+\Vert h'\Vert_{L^{\infty}[0,T]})(\eta^\epsilon(T)+T\eta^\epsilon(T))
\nonumber
\\
&\qquad
+\sup_{0\leq t\leq T}|Z_{t}^{\epsilon}-\eta^\epsilon(t)|
\left[|\log\phi(0)|
+\frac{\alpha}{\inf_{x\geq 0}\phi(x)}\Vert h\Vert_{L^{\infty}[0,T]}\eta^\epsilon(T)\right],
\nonumber
\end{align}
and secondly,
\begin{align}\label{estimate2}
&\left|\int_{0}^{T}\log\left(\phi\left(\int_{0}^{t-} h(t-s)dZ_{s}^{\epsilon}\right)\right)dZ_{t}^{\epsilon}
-\int_{0}^{T}\log\left(\phi\left(\int_{0}^{t} h(t-s)d\eta^\epsilon(s)\right)\right)dZ_{t}^{\epsilon}\right|
\\
&\leq\frac{\alpha}{\inf_{x\geq 0}\phi(x)}\left[|h(0)|+\int_{0}^{T}|h'(t)|dt\right]
\sup_{0\leq t\leq T}|Z_{t}^{\epsilon}-\eta^\epsilon(t)|
[\eta^\epsilon(T)+\delta],
\nonumber
\end{align}
and thirdly,
\begin{align}\label{estimate3}
&\left|\int_{0}^{T} \phi\left(\int_{0}^{t}h(t-s)dZ_{s}^{\epsilon}\right) ds
-\int_{0}^{T}\ \phi\left(\int_{0}^{t}h(t-s)d\eta^\epsilon(s)\right)ds\right|
\\
&\leq
T\alpha\left[|h(0)|+\int_{0}^{T}|h'(t)|dt\right]
\sup_{0\leq t\leq T}|Z_{t}^{\epsilon}-\eta^\epsilon(t)|,
\nonumber
\end{align}
and finally,
\begin{align}\label{estimate4}
&\left|\int_{0}^{T}\log\left(\phi\left(\int_{0}^{t} h(t-s)dZ^0(s)\right)\right)dZ_{t}^{\epsilon}
-\int_{0}^{T}\log\left(\phi\left(\int_{0}^{t} h(t-s)dZ^0(s)\right)\right)d\eta^\epsilon(t)\right|
\\
&\leq
\sup_{0\leq t\leq T}|Z_{t}^{\epsilon}-\eta^\epsilon(t)|\frac{\alpha}{\inf_{x\geq 0}\phi(x)}
(|h(0)|+\Vert h'\Vert_{L^{\infty}[0,T]})(Z^0(T)+TZ^0(T))
\nonumber
\\
&\qquad
+\sup_{0\leq t\leq T}|Z_{t}^{\epsilon}-\eta^\epsilon(t)|
\left[|\log\phi(0)|
+\frac{\alpha}{\inf_{x\geq 0}\phi(x)}\Vert h\Vert_{L^{\infty}[0,T]}Z^0(T)\right].
\nonumber
\end{align}
Thus, by \eqref{estimate1}, \eqref{estimate2}, \eqref{estimate3}, and \eqref{estimate4}, 
there exists a constant $C$ such that for any $\sup_{0\leq t\leq T}|Z_t^{\epsilon}-\eta^{\epsilon}(t)|\leq\delta(\epsilon)$,
$$
\begin{aligned}
 &\Bigg|\Bigg[\frac{1}{\epsilon}\int_{0}^{T}\log\left( \frac{\phi\left(\int_{0}^{t-} h(t-s)dZ_{s}^{\epsilon}\right)}{\phi\left(\int_{0}^{t} h(t-s)dZ_s^0\right)}\right)d Z_{t}^{\epsilon}
 \\
 &\qquad\qquad-\frac{1}{\epsilon} \int_{0}^{T} \left(\phi\left(\int_{0}^{t}  h(t-s)dZ_{s}^{\epsilon}\right)-\phi\left(\int_{0}^{t}  h(t-s)dZ_s^0\right)\right) ds\Bigg]\\
&\qquad - \Bigg[ \int_{0}^{T}\log\left( \frac{\phi\left(\int_{0}^{t} h(t-s)d\eta^{\epsilon}(s)\right)}{\phi\left(\int_{0}^{t} h(t-s)dZ_s^0\right)}\right)d \eta^{\epsilon}(t)
\\
&\qquad\qquad - \int_{0}^{T} \left(\phi\left(\int_{0}^{t} h(t-s)d\eta^{\epsilon}(s)\right)-\phi\left(\int_{0}^{t}  h(t-s)dZ_s^0\right)\right) ds\Bigg]\Bigg|\\
 \leq & \frac{\delta a^2(\epsilon) C}{\epsilon}.
 \end{aligned}
$$

Moreover, by a deterministic time change, we have
\begin{equation}
Z^{\epsilon}_{t}-Z^{0}_{t}=Y^{\epsilon}(Z_{t}^{0})=Y^{\epsilon}\left(\int_{0}^{t}\phi\left(\int_{0}^{s} h(s-u)dZ_u^0\right)ds\right),
\end{equation}
where the first equality above holds in distribution, 
where  $Y^{\epsilon}(t):=\epsilon\bar{N}_{t}^{\epsilon}-t$, and $\bar{N}_{t}^{\epsilon}$ is a standard Poisson process
with constant intensity $\frac{1}{\epsilon}$ under the probability measure $\tilde{\mathbb{P}}$.

It is well known that, see e.g. \cite{Mog}, that $\tilde{\mathbb{P}}(\frac{Y^{\epsilon}}{a(\epsilon)}\in\cdot)$ satisfies
a large deviation principle on $D[0,Z_{T}^{0}]$, see e.g. \cite{Mog}, with the speed $\frac{a^{2}(\epsilon)}{\epsilon}$ and the rate function
\begin{equation}
J_{Pos}(\xi)=
\begin{cases}
\frac{1}{2}\int_{0}^{Z_{T}^{0}}|\xi'(t)|^{2}dt  &\text{if $\xi\in\mathcal{AC}_{0}[0,Z_{T}^{0}]$},
\\
+\infty &\text{otherwise}.
\end{cases}
\end{equation}
 
Therefore,
\begin{align*}
&\lim_{\delta \rightarrow 0}\lim_{\epsilon\rightarrow 0} \frac{\epsilon}{a^2(\epsilon)}
\log\tilde{\mathbb{P}}\left(\sup_{0\leq t\leq T}|Z_t^{\epsilon}-\eta^{\epsilon}(t)|\leq\delta(\epsilon)\right)
\\
&=\lim_{\delta \rightarrow 0}\lim_{\epsilon\rightarrow 0} \frac{\epsilon}{a^2(\epsilon)}
\log\tilde{\mathbb{P}}\left(\sup_{0\leq t\leq T}|Z_t^{\epsilon}-Z_{t}^{0}-a(\epsilon)\eta(t)|\leq\delta a(\epsilon)\right)
\\
&=\lim_{\delta \rightarrow 0}\lim_{\epsilon\rightarrow 0} \frac{\epsilon}{a^2(\epsilon)}
\log\tilde{\mathbb{P}}\left(\sup_{0\leq t\leq T}|Y^{\epsilon}(Z_{t}^{0})-a(\epsilon)\eta(t)|\leq\delta a(\epsilon)\right)
\\
&=\lim_{\delta \rightarrow 0}\lim_{\epsilon\rightarrow 0} \frac{\epsilon}{a^2(\epsilon)}
\log\tilde{\mathbb{P}}\left(\sup_{0\leq t\leq Z_{T}^{0}}|Y^{\epsilon}(t)-a(\epsilon)\xi(t)|\leq\delta a(\epsilon)\right)
\\
&=J_{Pos}(\xi),
\end{align*}
where $\xi(\cdot)$ is defined via $\xi(Z_{t}^{0})=\eta(t)$, for every $0\leq t\leq T$, so that if $\eta\not\in \mathcal {AC}_0[0,T]$, then  $J_{Pos}(\xi)=+\infty$, and if $\eta\in \mathcal {AC}_0[0,T]$, then $\eta'(t)=(Z_{t}^{0})'\xi'(Z_{t}^{0})$ and
\begin{equation*}
\int_{0}^{Z_{T}^{0}}|\xi'(t)|^{2}dt
=\int_{0}^{T}|\xi'(Z_{t}^{0})|^{2}dZ_{t}^{0}
=\int_{0}^{T}\frac{|\eta'(t)|^{2}}{(Z_{t}^{0})'}dt
=\int_{0}^{T} \frac{|\eta'(t)|^2}{\phi\left(\int_{0}^{t} h(t-s)dZ_s^0\right)}dt.
\end{equation*}

Thus, we have
\begin{align*}
&\lim_{\delta \rightarrow 0}\lim_{\epsilon\rightarrow 0} \frac{\epsilon}{a^2(\epsilon)}
\log\tilde{\mathbb{P}}\left(\sup_{0\leq t\leq T}|Z_t^{\epsilon}-\eta^{\epsilon}(t)|\leq\delta(\epsilon)\right)
\\
&=\begin{cases}
-\frac{1}{2}\int_{0}^{T} \frac{|\eta'(t)|^2}{\phi\left(\int_{0}^{t} h(t-s)dZ_s^0\right)}dt  
&\text{if $\xi\in\mathcal{AC}_{0}[0,Z_{T}^{0}]$},
\\
-\infty &\text{otherwise}.
\end{cases}      
\end{align*}

Finally, we notice that
$$\begin{aligned}
&  \lim_{\epsilon\rightarrow 0}\frac{1}{a^2(\epsilon)}\int_{0}^{T}\log\left(\frac{\phi\left(\int_{0}^{t}h(t-s)d\eta^{\epsilon}(s)\right)}{\phi\left(\int_{0}^{t}  h(t-s)dZ_s^0\right)}\right)d\eta^{\epsilon}(t) \\
&\qquad-\lim_{\epsilon\rightarrow 0}\frac{1}{a^2(\epsilon)}\int_{0}^{T} \left(\phi\left(\int_{0}^{t} h(t-s)d\eta^{\epsilon}(s)\right)-\phi\left(\int_{0}^{t}  h(t-s)dZ_s^0\right)\right) dt\\
 &=\int_{0}^{T}\frac{\phi'\left(\int_{0}^{t}h(t-s)dZ_s^0\right)\int_{0}^{t}h(t-s)d\eta(s)}{\phi\left(\int_{0}^{t}  h(t-s)dZ_s^0\right)}d\eta(t)\\
  &\qquad -\frac{1}{2}\int_{0}^{T}\frac{|\phi'\left(\int_{0}^{t}h(t-s)dZ_s^0\right)\int_{0}^{t}h(t-s)d\eta(s)|^2}{\phi\left(\int_{0}^{t}  h(t-s)dZ_s^0\right)} dt.
\\
\end{aligned}$$
Hence, we conclude that
$$\begin{aligned}
&\lim_{\delta \rightarrow 0}\lim_{\epsilon\rightarrow 0}
\frac{\epsilon}{a^2(\epsilon)}
\log\tilde{\mathbb{E}}\Bigg[e^{ \frac{1}{\epsilon}\int_{0}^{T}\log\left( \frac{\phi\left(\int_{0}^{t-}  h(t-s)dZ_{s}^{\epsilon}\right)}{\phi\left(\int_{0}^{t} h(t-s)dZ_s^0\right)}\right)d Z_{t}^{\epsilon}
 -\frac{1}{\epsilon} \int_{0}^{T} \left(\phi\left(\int_{0}^{t}  h(t-s)dZ_{s}^{\epsilon}\right)-\phi\left(\int_{0}^{t}  h(t-s)dZ_s^0\right)\right) dt}\\
&\qquad\qquad\qquad\qquad 1_{\sup_{0\leq t\leq T}|Z_t^{\epsilon}-\eta^{\epsilon}(t)|\leq\delta(\epsilon)}\Bigg]\\
&= \lim_{\epsilon\rightarrow 0}\frac{1}{a^2(\epsilon)}\int_{0}^{T}\log\left(\frac{\phi\left(\int_{0}^{t}h(t-s)d\eta^{\epsilon}(s)\right)}{\phi\left(\int_{0}^{t}  h(t-s)dZ_s^0\right)}\right)d\eta^{\epsilon}(t) \\
&\qquad\qquad-\lim_{\epsilon\rightarrow 0}\frac{1}{a^2(\epsilon)}\int_{0}^{T} \left(\phi\left(\int_{0}^{t} h(t-s)d\eta^{\epsilon}(s)\right)-\phi\left(\int_{0}^{t}  h(t-s)dZ_s^0\right)\right) dt\\
&\qquad\qquad\quad+\lim_{\delta \rightarrow 0}\lim_{\epsilon\rightarrow 0} \frac{\epsilon}{a^2(\epsilon)}
\log\tilde{\mathbb{P}}\left(\sup_{0\leq t\leq T}|Z_t^{\epsilon}-\eta^{\epsilon}(t)|\leq\delta(\epsilon)\right)\\
&=-J(\eta).
\end{aligned}$$

Hence, the conclusion of the Theorem \ref{localMDPThm} holds.
\end{proof}

%%%%%%%%%%%%%%%%%%%%%%%%%%%%%%%%%%%%%%%%%%%%%%%%%%%%%%%%%%%%%%%%%%%%
%%%%%%%%%%%%%%%%%%%%%%%%%%%%%%%%%%%%%%%%%%%%%%%%%%%%%%%%%%%%
\begin{proof}[Proof of Lemma \ref{MDPexptightI}]
Let us recall that $Z_{t}^{\epsilon}$ satisfies the dynamics:
\begin{align}
Z_{t}^{\epsilon}
&=\int_{0}^{t}\int_{0}^{\infty}1_{[0,\phi(\int_{0}^{s}h(s-u)dZ_{u}^{\epsilon})]}(z)\epsilon\pi^{\epsilon^{-1}}(dzds)
\\
&=M_{t}^{\epsilon}+\int_{0}^{t}\phi\left(\int_{0}^{s}h(s-u)dZ_{u}^{\epsilon}\right)ds,
\nonumber
\end{align}
where $M_{t}^{\epsilon}$ is a martingale. 
Therefore, for any $0\leq t\leq T$,
\begin{align}
|Z_{t}^{\epsilon}-Z_{t}^{0}|
&\leq\sup_{0\leq t\leq T}|M_{t}^{\epsilon}|
+\int_{0}^{t}\left|\phi\left(\int_{0}^{s}h(s-u)dZ_{u}^{\epsilon}\right)-
\phi\left(\int_{0}^{s}h(s-u)dZ_{u}^{0}\right)\right|ds
\\
&\leq
\sup_{0\leq t\leq T}|M_{t}^{\epsilon}|
+\alpha\left[|h(0)|+\int_{0}^{T}|h'(t)|dt\right]
\int_{0}^{t}\sup_{0\leq u\leq s}|Z_{u}^{\epsilon}-Z_{u}^{0}|ds.
\nonumber
\end{align}
It follows from Gronwall's inequality that
\begin{equation}
\sup_{0\leq t\leq T}|Z_{t}^{\epsilon}-Z_{t}^{0}|
\leq\sup_{0\leq t\leq T}|M_{t}^{\epsilon}|
e^{\alpha\left[|h(0)|+\int_{0}^{T}|h'(t)|dt\right]T}.
\end{equation}
For any $\theta>0$, by Doob's martingale inequality,
\begin{align}
&\mathbb{P}\left(\sup_{0\leq t\leq T}|Z_{t}^{\epsilon}-Z_{t}^{0}|\geq Ka(\epsilon)\right)
\\
&\leq\mathbb{P}\left(\sup_{0\leq t\leq T}|M_{t}^{\epsilon}|
e^{\alpha\left[|h(0)|+\int_{0}^{T}|h'(t)|dt\right]T}\geq Ka(\epsilon)\right)
\nonumber
\\
&\leq\mathbb{P}\left(\sup_{0\leq t\leq T}M_{t}^{\epsilon}
e^{\alpha\left[|h(0)|+\int_{0}^{T}|h'(t)|dt\right]T}\geq Ka(\epsilon)\right)
\nonumber
\\
&\qquad\qquad\qquad
+\mathbb{P}\left(\sup_{0\leq t\leq T}(-M_{t}^{\epsilon})
e^{\alpha\left[|h(0)|+\int_{0}^{T}|h'(t)|dt\right]T}\geq Ka(\epsilon)\right)
\nonumber
\\
&\leq\mathbb{E}\left[e^{\frac{a(\epsilon)}{\epsilon}\theta
e^{\alpha\left[|h(0)|+\int_{0}^{T}|h'(t)|dt\right]T}M_{T}^{\epsilon}}\right]e^{-\theta K\frac{a(\epsilon)^{2}}{\epsilon}}
+\mathbb{E}\left[e^{-\frac{a(\epsilon)}{\epsilon}\theta
e^{\alpha\left[|h(0)|+\int_{0}^{T}|h'(t)|dt\right]T}M_{T}^{\epsilon}}\right]e^{-\theta K\frac{a(\epsilon)^{2}}{\epsilon}}.
\nonumber
\end{align}
Let us define
\begin{equation}
R_{t}^{\epsilon}=\frac{a(\epsilon)}{\epsilon}\theta
e^{\alpha\left[|h(0)|+\int_{0}^{T}|h'(t)|dt\right]T}M_{t}^{\epsilon}.
\end{equation}
Then $R_{t}^{\epsilon}$ is a martingale and $R_{0}^{\epsilon}=0$.
Moreover, $|\Delta M^{\epsilon}|\leq\epsilon$
and 
\begin{equation}
|\Delta R^{\epsilon}|\leq c:=a(\epsilon)\theta
e^{\alpha\left[|h(0)|+\int_{0}^{T}|h'(t)|dt\right]T}.
\end{equation}
By Lemma 26.19 in Kallenberg \cite{Kallenberg}, 
if $M$ is a local martingale starting at $0$ with $|\Delta M|\leq c$
then $e^{M-b\langle M\rangle}$ is a supermartingale 
where $b=g(c):=(e^{c}-1-c)c^{-2}$. Hence,
\begin{align}
\mathbb{E}[e^{R_{T}^{\epsilon}}]
&=\mathbb{E}[e^{R_{T}^{\epsilon}-\frac{1}{2}g(2c)\langle 2R^{\epsilon}\rangle_{T}}e^{\frac{1}{2}g(2c)\langle R^{\epsilon}\rangle_{T}}]
\\
&\leq\mathbb{E}\left[e^{2R_{T}^{\epsilon}-g(2c)\langle 2R^{\epsilon}\rangle_{T}}\right]^{1/2}
\mathbb{E}\left[e^{g(2c)\langle 2R^{\epsilon}\rangle_{T}}\right]^{1/2}
\nonumber
\\
&\leq\mathbb{E}\left[e^{4g(2c)\langle R^{\epsilon}\rangle_{T}}\right]^{1/2}.
\nonumber
\end{align}
Similarly,
\begin{equation}
\mathbb{E}[e^{-R_{T}^{\epsilon}}]
\leq\mathbb{E}\left[e^{4g(2c)\langle R^{\epsilon}\rangle_{T}}\right]^{1/2}.
\end{equation}
As $\epsilon\rightarrow 0$, it is easy to see that $c\rightarrow 0$
and $g(2c)\rightarrow\frac{1}{2}$. 
Therefore, for sufficiently small $\epsilon$,
\begin{align}\label{Rineq}
\mathbb{E}[e^{R_{T}^{\epsilon}}]
&\leq\mathbb{E}\left[e^{4\langle R^{\epsilon}\rangle_{T}}\right]^{1/2}
\\
&=\mathbb{E}\left[e^{4\frac{a(\epsilon)^{2}}{\epsilon^{2}}\theta^{2}
e^{2\alpha\left[|h(0)|+\int_{0}^{T}|h'(t)|dt\right]T}
\epsilon\int_{0}^{T}\phi(\int_{0}^{t}h(t-s)dZ_{s}^{\epsilon})dt}\right]^{1/2}
\nonumber
\\
&\leq\mathbb{E}\left[e^{4\frac{a(\epsilon)^{2}}{\epsilon^{2}}\theta^{2}
e^{2\alpha\left[|h(0)|+\int_{0}^{T}|h'(t)|dt\right]T}
\epsilon(\phi(0)T+\alpha\epsilon\int_{0}^{T}\int_{0}^{t}|h(t-s)|dN_{s}^{\epsilon}dt)}\right]^{1/2}
\nonumber
\\
&\leq
e^{2\frac{a(\epsilon)^{2}}{\epsilon}\theta^{2}
e^{2\alpha\left[|h(0)|+\int_{0}^{T}|h'(t)|dt\right]T}\phi(0)T}
\mathbb{E}\left[e^{4a(\epsilon)^{2}\theta^{2}
e^{2\alpha\left[|h(0)|+\int_{0}^{T}|h'(t)|dt\right]T}
\alpha\Vert h\Vert_{L^{\infty}[0,T]}TN_{T}^{\epsilon}}\right]^{1/2}
\nonumber
\\
&\leq
e^{2\frac{a(\epsilon)^{2}}{\epsilon}\theta^{2}
e^{2\alpha\left[|h(0)|+\int_{0}^{T}|h'(t)|dt\right]T}\phi(0)T}
e^{\frac{1}{2}C\left(4a(\epsilon)^{2}\theta^{2}
e^{2\alpha\left[|h(0)|+\int_{0}^{T}|h'(t)|dt\right]T}
\alpha\Vert h\Vert_{L^{\infty}[0,T]}T\right)\frac{1}{\epsilon}},
\nonumber
\end{align}
where $C(\theta)$ for small $\theta>0$ is defined in the proof of 
the exponential tightness for the large deviation principle
under the assumption that $\alpha\Vert h\Vert_{L^{1}[0,T]}<1$
and the fact that $a(\epsilon)^{2}$ is sufficiently small as $\epsilon\rightarrow 0$.
It is easy to check that for some $\bar{C}>0$
\begin{equation}
\bar{C}:=\limsup_{\theta\rightarrow 0}\frac{C(\theta)}{\theta}<\infty.
\end{equation}
Therefore, we conclude that
\begin{align}
&\limsup_{\epsilon\rightarrow 0}
\frac{\epsilon}{a(\epsilon)^{2}}\log\mathbb{P}\left(\sup_{0\leq t\leq T}|Z_{t}^{\epsilon}-Z_{t}^{0}|\geq Ka(\epsilon)\right)
\\
&\leq
2\theta^{2}
e^{2\alpha\left[|h(0)|+\int_{0}^{T}|h'(t)|dt\right]T}\phi(0)T
+2\theta^{2}
e^{2\alpha\left[|h(0)|+\int_{0}^{T}|h'(t)|dt\right]T}
\alpha\Vert h\Vert_{L^{\infty}[0,T]}T\bar{C}
-\theta K.
\nonumber
\end{align}
The desired result follows by letting $K\rightarrow\infty$.
\end{proof}

%%%%%%%%%%%%%%%%%%%%%%%%%%%%%%%%%%%%%%%%%%%%%%%%%%%%%%%%%%%%%%%%%
\begin{proof}[Proof of Lemma \ref{MDPexptightII}]
Without loss of generality, let us assume that $MT\in\mathbb{N}$.
For any $\delta>0$,
\begin{align} 
&\mathbb{P}\left(\sup_{0\leq s\leq t\leq T,
|t-s|\leq\frac{1}{M}}|Z_{t}^{\epsilon}-Z_{t}^{0}-Z_{s}^{\epsilon}+Z_{s}^{0}|\geq\delta a(\epsilon)\right)
\\
&\leq
\mathbb{P}\left(\exists j, 1\leq j\leq MT:
\left|Z_{\frac{j}{M}}^{\epsilon}-Z_{\frac{j-1}{M}}^{\epsilon}
-Z_{\frac{j}{M}}^{0}+Z_{\frac{j-1}{M}}^{0}\right|\geq\frac{\delta a(\epsilon)}{2}\right)
\\
&\leq
\sum_{j=1}^{MT}\mathbb{P}\left(
\left|Z_{\frac{j}{M}}^{\epsilon}-Z_{\frac{j-1}{M}}^{\epsilon}
-Z_{\frac{j}{M}}^{0}+Z_{\frac{j-1}{M}}^{0}\right|\geq\frac{\delta a(\epsilon)}{2}\right).
\nonumber
\end{align}
We can estimate that
\begin{align}
&\left|Z_{\frac{j}{M}}^{\epsilon}-Z_{\frac{j-1}{M}}^{\epsilon}
-Z_{\frac{j}{M}}^{0}+Z_{\frac{j-1}{M}}^{0}\right|
\\
&\leq|M_{j/M}^{\epsilon}-M_{(j-1)/M}^{\epsilon}|
+\int_{\frac{j-1}{M}}^{\frac{j}{M}}\left|
\phi\left(\int_{0}^{s}h(s-u)dZ_{u}^{\epsilon}\right)
-\phi\left(\int_{0}^{s}h(s-u)dZ_{u}^{0}\right)\right|ds
\nonumber
\\
&\leq|M_{j/M}^{\epsilon}-M_{(j-1)/M}^{\epsilon}|
+\frac{1}{M}\alpha\left[|h(0)|+\int_{0}^{T}|h'(t)|dt\right]
\sup_{0\leq t\leq T}|Z_{t}^{\epsilon}-Z_{t}^{0}|.
\nonumber
\end{align}
Thus,
\begin{align}
&\mathbb{P}\left(
\left|Z_{\frac{j}{M}}^{\epsilon}-Z_{\frac{j-1}{M}}^{\epsilon}
-Z_{\frac{j}{M}}^{0}+Z_{\frac{j-1}{M}}^{0}\right|\geq\frac{\delta a(\epsilon)}{2}\right)
\\
&\leq
\mathbb{P}\left(
\left|M_{j/M}^{\epsilon}-M_{(j-1)/M}^{\epsilon}\right|\geq\frac{\delta a(\epsilon)}{2}\right)
\nonumber
\\
&\qquad\qquad\qquad
+\mathbb{P}\left(
\frac{1}{M}\alpha\left[|h(0)|+\int_{0}^{T}|h'(t)|dt\right]
\sup_{0\leq t\leq T}|Z_{t}^{\epsilon}-Z_{t}^{0}|\geq\frac{\delta a(\epsilon)}{2}\right).
\nonumber
\end{align}
We can compute that for any $\theta>0$, for sufficiently small $\epsilon>0$,
\begin{align}
&\mathbb{P}\left(
\left|M_{j/M}^{\epsilon}-M_{(j-1)/M}^{\epsilon}\right|\geq\frac{\delta a(\epsilon)}{2}\right)
\\
&\leq\mathbb{P}\left(
\sup_{(j-1)/M\leq t\leq j/M}(M_{t}^{\epsilon}-M_{(j-1)/M}^{\epsilon})\geq\frac{\delta a(\epsilon)}{2}\right)
\nonumber
\\
&\qquad\qquad
+\mathbb{P}\left(
\sup_{(j-1)/M\leq t\leq j/M}(-M_{t}^{\epsilon}+M_{(j-1)/M}^{\epsilon})\geq\frac{\delta a(\epsilon)}{2}\right)
\nonumber
\\
&\leq\left(\mathbb{E}\left[e^{\theta\frac{a(\epsilon)}{\epsilon}(M_{j/M}^{\epsilon}-M_{(j-1)/M}^{\epsilon})}\right]
+\mathbb{E}\left[e^{\theta\frac{1}{\delta}\frac{a(\epsilon)}{\epsilon}(M_{(j-1)/M}^{\epsilon}-M_{j/M}^{\epsilon})}\right]\right)
e^{-\theta\frac{\delta}{2}\frac{a(\epsilon)^{2}}{\epsilon}}
\nonumber
\\
&\leq 
2\mathbb{E}\left[e^{4\theta^{2}
\frac{a(\epsilon)^{2}}{\epsilon^{2}}
\epsilon\int_{\frac{j-1}{M}}^{\frac{j}{M}}\phi(\int_{0}^{s}h(s-u)dZ_{u}^{\epsilon})ds}\right]^{1/2}
e^{-\theta\frac{\delta}{2}\frac{a(\epsilon)^{2}}{\epsilon}},
\nonumber
\end{align}
where the last line uses \eqref{Rineq}. From here, we can further estimate that
\begin{align}
&\mathbb{P}\left(
\left|M_{j/M}^{\epsilon}-M_{(j-1)/M}^{\epsilon}\right|\geq\frac{\delta a(\epsilon)}{2}\right)
\\
&\leq
2e^{4\theta^{2}
\frac{a(\epsilon)^{2}}{\epsilon}\frac{1}{M}\phi(0)}
\mathbb{E}\left[e^{2\theta^{2}
a(\epsilon)^{2}\frac{1}{M}\alpha\Vert h\Vert_{L^{\infty}[0,T]}N_{T}^{\epsilon}}\right]^{1/2}
e^{-\theta\frac{\delta}{2}\frac{a(\epsilon)^{2}}{\epsilon}},
\nonumber
\end{align}
which is uniform in $j$. Moreover,
\begin{align}
&\limsup_{\epsilon\rightarrow 0}\frac{\epsilon}{a(\epsilon)^{2}}
\log\mathbb{P}\left(
\left|M_{j/M}^{\epsilon}-M_{(j-1)/M}^{\epsilon}\right|\geq\frac{\delta a(\epsilon)}{2}\right)
\\
&\leq
2\theta^{2}\frac{1}{M}\phi(0)
+2\theta^{2}\bar{C}\frac{1}{M}\alpha\Vert h\Vert_{L^{\infty}[0,T]}
-\theta\frac{\delta}{2}.
\nonumber
\end{align}
The choice of $\theta>0$ is arbitrary. Let us choose $\theta=\sqrt{M}$, then
\begin{equation}
\limsup_{M\rightarrow\infty}
\limsup_{\epsilon\rightarrow 0}\frac{\epsilon}{a(\epsilon)^{2}}
\log\mathbb{P}\left(
\left|M_{j/M}^{\epsilon}-M_{(j-1)/M}^{\epsilon}\right|\geq\frac{\delta a(\epsilon)}{2}\right)
=-\infty.
\end{equation}
Finally, by Lemma \ref{MDPexptightI}, 
\begin{equation}
\limsup_{M\rightarrow\infty}
\limsup_{\epsilon\rightarrow 0}\frac{\epsilon}{a(\epsilon)^{2}}
\log\mathbb{P}\left(
\frac{1}{M}\alpha\left[|h(0)|+\int_{0}^{T}|h'(t)|dt\right]
\sup_{0\leq t\leq T}|Z_{t}^{\epsilon}-Z_{t}^{0}|\geq\frac{\delta a(\epsilon)}{2}\right)
=-\infty.
\end{equation}
Hence, we have proved the desired result.
\end{proof}

%%%%%%%%%%%%%%%%%%%%%%%%%%%%%%%%%%%%%%%%%%%%%%%%%%%%%%%%%%%%%%%%%%%%

\section*{Acknowledgements}
We are very grateful to the Associate Editor and two anonymous referees
for their helpful comments and suggestions.
Fuqing Gao acknowledges support from NSFC Grant 11571262 
and the Specialized Research Fund for the Doctoral Program of Higher Education of China (Grant No. 20130141110076).
Lingjiong Zhu is grateful to the support from NSF Grant DMS-1613164.

%%%%%%%%%%%%%%%%%%%%%%%%%%%%%%%%%%%%%%%%%%%%%%%%%%%%%%%%%%%%%%%%%%%%

\end{document}